\documentclass[10pt]{amsart}
\usepackage{amsmath,amssymb,mathtools}
\usepackage[T1]{fontenc}
\usepackage{hyperref}
\hypersetup{hidelinks}
\newtheorem{theorem}{Theorem}[section]
\newtheorem{proposition}[theorem]{Proposition}
\newtheorem{lemma}[theorem]{Lemma}
\newtheorem{corollary}[theorem]{Corollary}
\theoremstyle{definition}
\newtheorem{definition}{Definition}[section]
\newtheorem{example}[theorem]{Example}
\theoremstyle{remark}

\newcommand{\C}{\mathbb C}
\newcommand{\R}{\mathbb R}
\newcommand{\Z}{\mathbb Z}
\newcommand{\PP}{\mathbb P}
\newcommand{\ii}{\sqrt{-1}}
\newcommand{\ddbar}{\partial\bar\partial}
\newcommand{\tr}{\operatorname{tr}}
\newcommand{\adj}{\operatorname{adj}}

\newcommand{\Span}{\operatorname{span}}
\newcommand{\Tot}{\operatorname{Tot}}
\newcommand{\Ad}{\operatorname{Ad}}
\newcommand{\rad}{\operatorname{rad}}
\newcommand{\aff}{\mathfrak{aff}}
\newcommand{\Lev}{\mathcal L}
\numberwithin{equation}{section}
\title[Gauduchon metrics]{Existence and uniqueness of conformal Gauduchon metrics}
\author{Hanzhang Yin}
\address{School of Mathematics, Harbin Institute of Technology, Harbin, Heilongjiang 150001, China}
\email{3199877769@qq.com}
\email{YinHZ@hit.edu.cn}
\thanks{The author is supported by the Research Start-up Fund of Harbin Institute of Technology.}
\date{}
\begin{document}
\begin{abstract}
In this paper, we prove that every Hermitian conformal class on a connected noncompact complex manifold
of dimension at least two contains a Gauduchon metric. We give a sufficient
condition for uniqueness of such Gauduchon metrics up to scaling in terms of functions bounded from
above and their Levi forms. This condition is necessary when the manifold
has at least two ends. In that case, uniqueness in a fixed conformal class
is equivalent to the assertion that every positive solution of the complex
Laplace equation for a metric in that class is constant.
For a connected noncompact complex Lie group of dimension $n$,
uniqueness in every conformal class is equivalent to the assertion that,
for every Hermitian metric, every positive solution of the complex Laplace
equation is constant. These properties hold exactly when the complex
tangent space of a maximal compact subgroup has dimension $n-1$; such a
group is abelian. In a conformal class containing a balanced metric, the
same equivalence holds without a restriction on the number of ends.
We also prove uniqueness for manifolds with a suitable potential on a
compactification with compact fibers over a point, the real line, or a
parabolic Riemann surface. The examples include compact pluripolar
deletions, partially deleted compact fibers, and semipositive line bundles.
We construct conformal Gauduchon metrics which are not related by scaling
on Stein manifolds and on the total space of $\mathcal O_{\PP^1}(1)^{\oplus2}$.
\end{abstract}
\maketitle

\section{Introduction}
Let $M$ be a connected noncompact complex manifold of complex
dimension $n\geq2$. A
Hermitian metric on a complex manifold is identified with its positive real
$(1,1)$-form $\omega$. The metric is
\emph{Gauduchon} when
\begin{equation}\label{eq:gauduchon}
\ddbar\omega^{n-1}=0.
\end{equation}
Here $\partial$ and $\bar\partial$ denote the $(1,0)$ and $(0,1)$ parts
of exterior differentiation, and the power denotes exterior
product. Two metrics are conformal if one is a smooth positive function
times the other. Uniqueness up to scaling means that two such Gauduchon
metrics differ by a positive constant.

Gauduchon proved that every Hermitian conformal class on a compact
connected complex manifold contains a Gauduchon metric, unique up to
scaling \cite{Gauduchon77}. The problem of prescribing the volume form, usually called
the Gauduchon conjecture, was solved by Sz\'ekelyhidi, Tosatti, and
Weinkove \cite{STW}. We study existence and uniqueness in conformal
classes on noncompact manifolds. Existence follows from the relation
between positive solutions of an elliptic equation and its formal adjoint.

\begin{theorem}\label{thm:existence}
Let $M$ be a connected noncompact complex manifold of complex dimension
$n\geq2$. For every smooth Hermitian metric $\alpha$ and every $x_0\in M$,
there is a smooth real function $\rho$ such that
\begin{equation}\label{eq:existence}
\rho>0,\qquad \rho(x_0)=1,\qquad
\ddbar(\rho\alpha^{n-1})=0.
\end{equation}
Then $\rho^{1/(n-1)}\alpha$ is a conformal Gauduchon metric.
\end{theorem}

The uniqueness question can be expressed in terms of the complex Hessian
of a bounded function. We use the following definition throughout.
\begin{definition}\label{def:positive-levi}
For a smooth real function $f$, its \emph{Levi form} at $p$ is the Hermitian
form on the holomorphic tangent space $T^{1,0}_pM$ with coordinate matrix
$(f_{z_j\bar z_k}(p))$. We denote this form by $\Lev_pf$;
$f_{z_j\bar z_k}=\partial^2f/(\partial z_j\partial\bar z_k)$.
We say that $f$ has a \emph{positive Levi direction
at every point} if, for each $p\in M$, there is a vector
$v\in T^{1,0}_pM$ such that $\Lev_pf(v,\bar v)>0$.
The direction may depend on $p$; no continuous choice of directions is
required.
\end{definition}
For each Hermitian metric $\alpha$, put
\begin{equation}\label{eq:trace-intro}
L_\alpha f=\tr_\alpha(\ii\ddbar f),\qquad P_\alpha=-L_\alpha.
\end{equation}
Here $\tr_\alpha$ denotes the trace of a $(1,1)$-form with respect to $\alpha$.
We call $L_\alpha$ the \emph{complex Laplacian}; it is also known as the
Chern Laplacian, with the sign and normalization in
\eqref{eq:trace-intro}. A \emph{positive solution of the complex Laplace
equation} is a smooth real function $u>0$ satisfying $L_\alpha u=0$.
Since $L_\alpha1=0$, uniqueness of positive solutions up to scaling means
that every positive solution is constant. The formal adjoint is defined
in Definition~\ref{def:elliptic} and is given for $P_\alpha$ by
\eqref{eq:adjoint}. Definition~\ref{def:elliptic} also explains when an
operator is critical. Since $P_\alpha1=0$, Proposition~\ref{prop:pinchover}
shows that $P_\alpha$ is critical if and only if every positive smooth
supersolution is constant.

An \emph{end} is a compatible choice of a connected component which is
not relatively compact outside a compact exhaustion. We say that $M$
has at least two ends if some compact set separates two components
whose closures in $M$ are noncompact.

We write $[\alpha]$ for the conformal class of a Hermitian metric $\alpha$.

\begin{theorem}[Levi criterion]\label{thm:levi}
Let $M$ be as in Theorem~\ref{thm:existence}. Consider the following
conditions:
\begin{enumerate}
\item There is no smooth real function on $M$ bounded from above with a
positive Levi direction at every point.
\item $P_\alpha$ is critical for every smooth Hermitian metric $\alpha$.
\item Every Hermitian conformal class contains a Gauduchon metric unique
up to scaling.
\item For every smooth Hermitian metric $\alpha$, every positive solution
of the complex Laplace equation $L_\alpha u=0$ is constant.
\end{enumerate}
Conditions (1) and (2) are equivalent and imply both (3) and (4).
If $M$ has at least two ends, all four conditions are equivalent.
More precisely, in that case, for each fixed Hermitian metric $\alpha$,
the following assertions are equivalent: $P_\alpha$ is critical;
the Gauduchon metric in $[\alpha]$ is unique up to scaling; and every
positive solution of $L_\alpha u=0$ is constant.
\end{theorem}

Compact complex directions give a concrete form of the uniqueness
criterion on Lie groups. A complex Lie group is a group with holomorphic
multiplication and inversion. Let $\mathfrak g=T_1G$ be its complex Lie
algebra. Here $1$ is the identity element of $G$, and $T_1G$ is its tangent
space at the identity.
A \emph{maximal compact subgroup} is a compact subgroup $K$ which
is contained in no larger compact subgroup. Its Lie algebra
$\mathfrak k=T_1K$ is a real subspace of the underlying real vector space
of $\mathfrak g$.
Thus
$\mathfrak k\cap\ii\mathfrak k$ is an intersection of real subspaces
inside $\mathfrak g$. It is stable under multiplication by $\ii$ and is
the largest complex subspace contained in $\mathfrak k$. Its complex
dimension is the CR dimension of $K$. Maximal compact subgroups are
conjugate, so this dimension is independent of the choice of $K$
\cite[p.~47]{KKO}.

\begin{theorem}[Lie group criterion]\label{thm:lie}
Let $G$ be a connected noncompact complex Lie group of complex dimension
$n\geq2$, and let $K$ be a maximal compact subgroup with real Lie algebra
$\mathfrak k$. The following conditions are equivalent:
\begin{enumerate}
\item For every Hermitian metric $\alpha$ on $G$, every positive solution
of the complex Laplace equation $L_\alpha u=0$ is constant.
\item Every Hermitian conformal class on $G$ contains a Gauduchon metric
unique up to scaling.
\item There is no smooth real function on $G$ bounded from above with a
positive Levi direction at every point.
\item The complex tangent space of $K$ has dimension $n-1$, that is,
\begin{equation}\label{eq:lie-criterion}
\dim_\C(\mathfrak k\cap\ii\mathfrak k)=n-1.
\end{equation}
\end{enumerate}
When these conditions hold, $G$ is abelian.
\end{theorem}
The metrics in this theorem are arbitrary smooth Hermitian metrics.
Completeness and invariance under the group are not assumed. The proof
uses Theorem~\ref{thm:levi} and Corollary~\ref{cor:proper-projection} in
the abelian case. Proposition~\ref{prop:nonabelian} treats nonabelian
groups by constructions using solvable and semisimple quotients. Examples are given in Section~\ref{sec:lie-examples}.

The argument using compact fibers also applies beyond Lie groups. A potential
which tends to $-\infty$ at the added boundary ensures that the fiber
minima in \eqref{eq:penalized} are attained. On a parabolic base, their
limiting function is superharmonic and hence constant.

\begin{theorem}[Compactifications with compact fibers]\label{thm:fibers}
Let $M$ be as in Theorem~\ref{thm:existence}. Suppose that it has a
compactification with compact fibers over a point, $\R$, or a parabolic
Riemann surface, with a plurisubharmonic potential at the added boundary,
in the sense of Definition~\ref{def:completion}. Write $\bar\pi:X\to B$
for the completed projection and $E=X\setminus M$ for the added boundary.
If the base is not a point, assume that there exists $s_0\in B$ such that
$E\cap\bar\pi^{-1}(s_0)=\varnothing$, so the whole completed fiber over
$s_0$ lies in $M$. Then every
$P_\alpha$ is critical, and every Hermitian conformal class on $M$
contains a Gauduchon metric unique up to scaling.
\end{theorem}
For a point base, the completed space is compact. For either nonconstant
base, only the completed fibers must be compact. The completed space need
not carry a complex structure.

Section~\ref{sec:preliminaries} recalls the elliptic and potential theoretic
facts used below. Section~\ref{sec:existence} proves existence through the
adjoint density equation. Section~\ref{sec:levi} proves the Levi criterion and its assertion
about Gauduchon metrics on manifolds with at least two ends, and gives
the Stein construction.
Section~\ref{sec:lie} proves the Lie group criterion.
Section~\ref{sec:ends} treats compactifications, changes on compact sets,
and their geometric applications. Section~\ref{sec:harmonic} compares
uniqueness of Gauduchon metrics with uniqueness of positive solutions
of the complex Laplace equation and completes the proofs of
Theorems~\ref{thm:levi} and~\ref{thm:lie}.

\section{Preliminaries}\label{sec:preliminaries}
We recall the analytic results which enter several proofs. All ellipticity
assumptions are local; no uniform bound at infinity is imposed.

\begin{definition}\label{def:elliptic}
Let $Q$ be a real second order operator on a connected noncompact smooth
manifold. In local real coordinates it has the form
\[
Q=-\sum_{a,b=1}^m a^{ab}(x)\partial_{x_a}\partial_{x_b}
  +\sum_{a=1}^m b^a(x)\partial_{x_a}+c(x),
\]
with smooth coefficients and positive definite principal matrix
$(a^{ab})$. Here $\partial_{x_a}$ denotes differentiation with respect to $x_a$.
A positive solution satisfies $u>0$ and $Qu=0$;
a positive supersolution satisfies $u>0$ and $Qu\geq0$.
The formal adjoint $Q^*$ relative to a smooth positive density $d\mu$
is characterized by
\begin{equation}\label{eq:formal-adjoint}
\int f(Qg)\,d\mu=\int(Q^*f)g\,d\mu
\end{equation}
for smooth compactly supported real functions $f,g$.
A positive minimal Green function is the least positive fundamental
solution with a given pole, obtained by exhaustion with Dirichlet Green
functions. An operator which has a positive solution is \emph{subcritical}
if this Green function exists, and \emph{critical} otherwise.
\end{definition}
These are the conventions of \cite[Definition~2.6 and Corollary~2.8]{Pinchover}.

\begin{proposition}[Pinchover {\cite{Pinchover}}]\label{prop:pinchover}
If $Q$ admits a positive solution, then it is either critical or
subcritical. It is critical exactly when its positive supersolutions
form one ray, meaning that any two are positive constant multiples of
each other. In that case its positive solutions form one ray as well.
The operators $Q$ and $Q^*$ are either both critical or both subcritical. In particular,
if $Q$ has a positive solution, then so does $Q^*$.
\end{proposition}
The statements are \cite[Definition~2.6 and Corollary~2.8(i),(ii),(iv)]{Pinchover}.
For positive solutions, \emph{minimal growth} means that a positive
solution which dominates the function on an inner comparison boundary
also dominates it throughout the exterior; see
\cite[Definition~2.4]{Pinchover}. The construction in
\cite[Theorem~2.5]{Pinchover} gives, in the subcritical case, smooth
nonnegative Green potentials of smooth nonnegative compactly supported
sources. Such a potential has minimal growth outside a compact set
containing its source.
Since $P_\alpha1=0$, Proposition~\ref{prop:pinchover} says that
$P_\alpha$ is critical exactly when all its positive smooth
supersolutions are constant.

We use the weak and strong maximum principles for smooth locally
elliptic operators \cite[Theorems~3.1 and 3.5]{GT}. In particular, a
positive smooth supersolution for $P_\alpha$ which attains its global
minimum is constant. Local elliptic regularity gives smoothness of the
positive adjoint solutions \cite[Chapter~7, \S4]{TaylorPDE}.

A function $\psi$ with values in $[-\infty,\infty)$ is
\emph{plurisubharmonic} if it is upper semicontinuous, is not identically
$-\infty$ on any component, and its restriction to every holomorphic
disc is subharmonic or identically $-\infty$. For a smooth function this
means $\ii\ddbar\psi\geq0$. We use the usual maximum, truncation, and
decreasing limit properties \cite[\S2.9]{Klimek}. A real smooth function
is \emph{pluriharmonic} if it is locally the real part of a holomorphic
function. A closed set is locally pluripolar if it is locally contained
in the pole set of a plurisubharmonic function. Locally bounded
plurisubharmonic functions extend across such a set
\cite[Theorem~1.8.6]{Forstneric}.

A connected noncompact Riemann surface is \emph{parabolic} if its
Laplacian has no positive Green function. Every nonnegative finite
superharmonic function on a parabolic surface is constant
\cite[Chapter~IV, \S1, Theorem~6C]{AS}. This includes $\C$ by the
subharmonic Liouville theorem \cite[Corollary~2.3.4]{Ransford}.
A map is \emph{proper} if inverse images of compact sets are compact.
For a continuous cost on an open surjection, compactness of its sublevels
over compact base sets implies attainment and continuity of the fiber
minimum \cite[Theorem~1.2]{FKV}.

Some fiber minima are only lower semicontinuous. We use the viscosity
interpretation: a smooth function $q$ touches $v$ from below at $s$ if
$q(s)=v(s)$ and $q\leq v$ nearby. The inequality $v_{z\bar z}\leq A$
means $q_{z\bar z}(s)\leq A(s)$ for all such tests. Upper tests define
the corresponding inequalities for upper semicontinuous functions.
For smooth linear elliptic operators these inequalities agree with the
distributional ones \cite[Theorems~9.2 and 9.3]{HL}.
If $v_j$ are locally bounded continuous functions with
$(v_j)_{z\bar z}\leq\varepsilon_j A$ locally, $\varepsilon_j\to0$, and
$A$ locally bounded, their finite lower relaxed limit
\begin{equation}\label{eq:relaxed}
V(s)=\liminf_{j\to\infty,\ t\to s}v_j(t)
\end{equation}
is superharmonic \cite[Lemma~6.1 and Remarks~6.2--6.3]{CIL}.
The analogous assertion on $\R$ gives $V''\leq0$, hence concavity.
Every nonnegative concave function on the whole real line is constant:
a nonzero secant slope forces negative values in one direction.

We need one further result about positive solutions on the ends.
Let $D$ be a smooth relatively compact domain, and set
$E=M\setminus\overline D$. For a subcritical $Q$, let $H_E f$ be the
minimal exterior Dirichlet solution with inner boundary data $f$,
obtained by taking zero outer data on a compact exhaustion. Write
$\mathcal C_Q(M)$ for the positive solution cone and put
\[
\mathcal D_Q(E)=\{h\geq0:Qh=0\text{ on }E,\ h|_{\partial D}=0\}
\setminus\{0\}.
\]
Zero trace means continuous zero boundary values; a function in this
cone may vanish on some exterior components.
\begin{proposition}[Taylor {\cite{TaylorMartin}}]\label{prop:taylor}
For subcritical $Q$, the map
\[
\mathcal C_Q(M)\longrightarrow\mathcal D_Q(E),\qquad
u\longmapsto u|_E-H_E(u|_{\partial D})
\]
is a positive homogeneous bijection. If $M$ has at least two ends,
$\mathcal C_Q(M)$ contains two nonproportional positive solutions.
\end{proposition}
The extension operator and its kernel are described in
\cite[\S5, Propositions~4 and 5]{TaylorMartin}; the inverse identity is
\cite[\S7, Lemma~3]{TaylorMartin}. The decomposition by ends and the
existence of minimal points in each boundary component follow from
\cite[\S8, Theorem~6 and Corollary~3]{TaylorMartin}. If $Q$ has a
zeroth order term, divide by a positive solution first, as in
\cite[\S7, Remark~1]{TaylorMartin}.

\section{The density equation and existence}\label{sec:existence}
The conformal equation is linear in the $(n-1)$st power of the conformal
factor. Its coordinate form identifies it with an adjoint elliptic
equation, which proves Theorem~\ref{thm:existence}.

In a holomorphic chart $z=(z_1,\ldots,z_n)$ write
\[
\alpha=\ii\sum_{j,k=1}^n g_{j\bar k}\,dz_j\wedge d\bar z_k,
\qquad G=(g_{j\bar k}).
\]
The matrix $G$ is smooth positive Hermitian. For a matrix field $A$ put
\begin{equation}\label{eq:divergence}
\mathcal D A=\sum_{j,k=1}^n\partial_{z_j}\partial_{\bar z_k}A_{kj}.
\end{equation}
The adjugate is $\adj G=(\det G)G^{-1}$. Expansion of the exterior
power, followed by $\partial\bar\partial$, gives
\begin{equation}\label{eq:density}
\ddbar(\rho\alpha^{n-1})=0
\quad\Longleftrightarrow\quad \mathcal D(\rho\adj G)=0.
\end{equation}
The cofactor index in \eqref{eq:density} is $kj$; see
\cite[\S8, equations~(19) and (23)]{GauduchonRicci}.
The trace in \eqref{eq:trace-intro} is
\begin{equation}\label{eq:chern}
L_\alpha f=\frac{n\ii\ddbar f\wedge\alpha^{n-1}}{\alpha^n}
=\sum_{j,k}(G^{-1})_{kj}f_{z_j\bar z_k}.
\end{equation}
The quotient of top degree forms means their scalar ratio. Relative to
$dV_\alpha=(2^n n!)^{-1}\alpha^n$, whose real coordinate density is
$\det G$, integration by parts gives
\begin{equation}\label{eq:adjoint}
P_\alpha^*\rho=-(\det G)^{-1}\mathcal D(\rho\adj G).
\end{equation}
All test functions have compact support, so this formula imposes no
boundary condition at infinity.

\begin{proposition}\label{prop:rays}
A positive smooth function $\rho$ satisfies $P_\alpha^*\rho=0$ exactly
when $\rho^{1/(n-1)}\alpha$ is Gauduchon. Uniqueness up to scaling of
Gauduchon metrics in the conformal class of $\alpha$ is equivalent to
proportionality of all positive solutions of the adjoint equation.
\end{proposition}
\begin{proof}
For $c>0$, $\adj(cG)=c^{n-1}\adj G$. Equations
\eqref{eq:density} and \eqref{eq:adjoint} give the first assertion.
Proportional densities give metric factors differing by a constant.
Conversely, taking the $(n-1)$st power of a constant metric ratio gives
proportional densities.
\end{proof}

\begin{proof}[Proof of Theorem~\ref{thm:existence}]
The identity $P_\alpha1=0$ and Proposition~\ref{prop:pinchover} give
a positive solution $v$ of $P_\alpha^*v=0$. Local elliptic regularity
makes $v$ smooth. Set $\rho=v/v(x_0)$. The equation is linear and
$v(x_0)>0$, so \eqref{eq:existence} follows from
\eqref{eq:adjoint}. Finally
$(\rho^{1/(n-1)}\alpha)^{n-1}=\rho\alpha^{n-1}$.
\end{proof}

We will also construct metrics by prescribing their $(n-1,n-1)$-forms.
For holomorphic frames $e$ and $f$, let $J$ be the change of frame matrix:
its columns are the vectors of $e$ expressed in $f$. Subscripts $e$ and $f$
indicate coefficient matrices in these frames. A superscript $T$ denotes
transpose, and a bar on a matrix denotes entrywise complex conjugation.
In cofactor coordinates, if the metric law is
$G_e=J^TG_f\bar J$, the corresponding law is
\begin{equation}\label{eq:cofactor-law}
A_e=\adj(\bar J)A_f\adj(J^T).
\end{equation}
Such matrix fields represent real $(n-1,n-1)$-forms, up to a fixed
nonzero dimensional factor. In this representation, the form is
annihilated by $\ddbar$ exactly when $\mathcal D A=0$.
\begin{lemma}\label{lem:root}
A smooth positive Hermitian cofactor field satisfying
\eqref{eq:cofactor-law} determines a smooth Hermitian metric with
coefficient matrix
\begin{equation}\label{eq:root}
G_A=(\det A)^{1/(n-1)}A^{-1},\qquad \adj G_A=A.
\end{equation}
\end{lemma}
\begin{proof}
Put $d=(\det A)^{1/(n-1)}$. Then $\det G_A=d$ and
$G_A^{-1}=d^{-1}A$, which proves the adjugate identity. Positivity and
smoothness follow from those of $A$ and its inverse. Substitution in
\eqref{eq:cofactor-law} gives the metric transformation law.
This is the coordinate version of Michelsohn's positive root
correspondence \cite[pp.~279--280, equations~(4.8)--(4.11)]{Michelsohn}.
\end{proof}
If $\mathcal D A=\mathcal D(hA)=0$ for a positive nonconstant smooth
function $h$, Lemma~\ref{lem:root} and \eqref{eq:density} give the two Gauduchon
metrics $G_A$ and $h^{1/(n-1)}G_A$, which are not related by scaling.

\section{The Levi criterion}\label{sec:levi}
A subcritical trace operator produces bounded Green potentials.
A convergent positive sum of them has strictly positive source
everywhere. This gives the bounded function needed in the Levi criterion.

\begin{lemma}\label{lem:green-bound}
For a subcritical $P_\alpha$, the minimal Green potential of a smooth
nonnegative compactly supported source is globally bounded.
\end{lemma}
\begin{proof}
Let $w\geq0$ be the potential. It is smooth and has minimal growth
outside a compact set containing its source, by
\cite[Theorem~2.5]{Pinchover}, as described after
Proposition~\ref{prop:pinchover}. On a larger compact
comparison domain it has a finite maximum. A positive constant $C$
above that maximum is a solution of $P_\alpha C=0$. Minimal growth gives
$w\leq C$ on the exterior as well.
\end{proof}

\begin{lemma}\label{lem:levi-function}
If $P_\alpha$ is subcritical for some Hermitian metric $\alpha$, then
$M$ has a smooth bounded function with a positive Levi direction at
every point.
\end{lemma}
\begin{proof}
Choose nonnegative smooth compactly supported functions $\rho_j$ whose
positive loci cover $M$, and let $w_j\geq0$ be their minimal potentials,
so $P_\alpha w_j=\rho_j$. Lemma~\ref{lem:green-bound} bounds each
$w_j$ globally. On a countable collection of relatively compact
coordinate windows covering $M$, take cutoff extensions of the $w_j$.
Enumerate the global bounds and all derivative suprema of these
extensions as $A_{mj}\geq0$, with the global bound first, and put
\begin{equation}\label{eq:series-coefficients}
c_j=\frac{2^{-j-2}}{1+\sum_{m=0}^j A_{mj}}.
\end{equation}
For each fixed $m$, $c_jA_{mj}\leq2^{-j-2}$ once $j\geq m$.
Consequently all derivative series converge locally, and
\[
w=\sum_{j=0}^\infty c_jw_j\in C^\infty(M),\qquad
0\leq w\leq\tfrac12,\qquad P_\alpha w=\sum_jc_j\rho_j>0.
\]
The convergence of the first and second derivative series justifies
the last equality. Set $f=1-w$. Then $1/2\leq f\leq1$ and
$L_\alpha f=P_\alpha w>0$. A Hermitian form with no positive
direction is negative semidefinite, and its trace against a positive
metric is nonpositive. Therefore the Levi form of $f$ has a positive
direction at every point.
\end{proof}

\begin{lemma}\label{lem:trace-selection}
If a smooth real function $f$ has a positive Levi direction at every
point, then there is a smooth Hermitian metric $\alpha$ for which
$L_\alpha f>0$ everywhere.
\end{lemma}
\begin{proof}
The Levi form is a smooth Hermitian form on $T^{1,0}M$.
Apply \cite[Proposition~1.3]{CHR} with bundle rank $d=n$, parameter
$q=n$, and empty prescribed closed set. The hypothesis requires one
positive eigenvalue; the conclusion makes the sum of all eigenvalues,
hence the metric trace, positive.
\end{proof}

\begin{proof}[Proof of Theorem~\ref{thm:levi}: conditions (1), (2), and Gauduchon metrics]
If no smooth function bounded from above has a positive Levi direction
at every point, Lemma~\ref{lem:levi-function} rules out a subcritical
$P_\alpha$. Since $P_\alpha1=0$, every $P_\alpha$ is critical.
Conversely, let $f\leq C$ have a positive Levi direction everywhere,
and choose $\alpha$ by Lemma~\ref{lem:trace-selection}. Then
$u=C+1-f>0$ and $P_\alpha u=L_\alpha f>0$. This is a positive
supersolution which is not a solution, so $P_\alpha$ is subcritical.
This proves the equivalence. In the critical case,
Proposition~\ref{prop:pinchover} gives one positive adjoint ray.
Proposition~\ref{prop:rays} gives uniqueness up to scaling, and
Theorem~\ref{thm:existence} gives existence. This proves the equivalence
of (1) and (2) and the implication (2)$\Rightarrow$(3).
\end{proof}

\phantomsection\label{proof:levi-metric-ends}
\begin{proof}[Proof of Theorem~\ref{thm:levi}: Gauduchon metrics when $M$ has at least two ends]
If $P_\alpha$ is critical, uniqueness follows from
Propositions~\ref{prop:pinchover}
and~\ref{prop:rays}. If $P_\alpha$ is subcritical, then so is
$P_\alpha^*$. Proposition~\ref{prop:taylor} gives two
nonproportional positive adjoint solutions on a manifold with at least
two ends. Their positive $(n-1)$st roots give two conformal Gauduchon
metrics which are not related by scaling. This proves the converse.
Applying this equivalence to every $\alpha$ proves
(3)$\Rightarrow$(2) when $M$ has at least two ends. The assertions about
positive solutions are proved in
\hyperref[proof:levi-positive-ends]{the remaining part of the proof below}.
\end{proof}

The following compact fiber consequence will be used for Lie groups.
\begin{corollary}\label{cor:proper-projection}
Suppose $p:M\to B$ is a proper surjection, where $B$ is a parabolic
Riemann surface or $\R$. Assume that $p$ is holomorphic in the first
case and pluriharmonic in the second. Then every Hermitian conformal
class on $M$ has a Gauduchon metric unique up to scaling, and every
$P_\alpha$ is critical.
\end{corollary}
\begin{proof}
Fix $u>0$ with $P_\alpha u\geq0$, and take its minimum $v$ on each
compact fiber. Properness gives attainment and lower semicontinuity.
A lower test $q$ for $v$ at $s$ makes $u-q\circ p$ attain a local
minimum at a fiber minimizer. At a regular point, the Hessian trace
gives $q_{z\bar z}\leq0$ on a curve base, or $q''\leq0$ on $\R$.
The same inequality holds at a critical minimizer: if the relevant
second derivative were positive, it would stay positive nearby.
The strong minimum principle would make $u-q\circ p$ constant
there, contradicting the strict trace inequality at a nearby regular
point. Nonconstant holomorphic maps to curves are open with dense
regular points; the local holomorphic representation gives the same
facts in the pluriharmonic case. Thus $v$ is a positive superharmonic
function on $B$, or a positive concave function on $\R$.
It is constant. Attainment on a fiber gives a global minimum of $u$,
so $u$ is constant by the strong minimum principle.
If a function $f$ bounded from above had a positive Levi direction everywhere,
Lemma~\ref{lem:trace-selection} would make $C+1-f$ a nonconstant
positive supersolution, a contradiction. Theorem~\ref{thm:levi}
now gives the assertions.
\end{proof}

\subsection{The Stein construction}
A Stein manifold may be defined by holomorphic convexity, separation
of points by global holomorphic functions, and global holomorphic
local coordinates. Holomorphic convexity means that the holomorphic
hull of each compact set is compact; global local coordinates mean
that $n$ global holomorphic functions have independent differentials
at each chosen point. We use the finite holomorphic immersion supplied
by the Stein embedding theorem \cite[p.~123]{EG}.

\begin{proposition}\label{prop:stein}
Every connected noncompact Stein manifold of dimension $n\geq2$
has two conformal Gauduchon metrics which are not related by scaling.
The construction uses a smooth nonconstant density $h$ with $0<h<2$.
\end{proposition}
We use $\Re z$ and $\Im z$ for the real and imaginary parts of a complex
number $z$. The notation $\operatorname{Re}z$ has the same meaning as $\Re z$.
On $\C^n$ write $x=\Re z_1$, $w=z_2$,
and let $E_{12}$ have its only nonzero entry, equal to one, at row~1,
column~2. Put
\begin{equation}\label{eq:stein-target}
a(x)=e^{x^2/2}\cosh x,\quad
T=I-\tfrac12xwE_{12},\quad
A=aT^*T,\quad h_0=1+\tfrac12\tanh x.
\end{equation}
Here $T^*$ denotes conjugate transpose.
The symbol $I$ denotes the identity matrix of the required size;
$I_r$ denotes the $r\times r$ identity matrix.
\begin{lemma}\label{lem:stein-target}
The matrix field $A$ is smooth positive Hermitian and satisfies
$\mathcal DA=\mathcal D(h_0A)=0$.
The function $s\mapsto1+\tfrac12\tanh s$ is strictly increasing
and lies between $1/2$ and $3/2$.
\end{lemma}
\begin{proof}
The triangular matrix $T$ is invertible, so $aT^*T$ is positive.
For a matrix of this shape write
\[
A_{11}=p(x),\quad A_{12}=q(x)w,\quad
A_{21}=q(x)\bar w,\quad A_{22}=p(x)+r(x)|w|^2,
\]
with $A_{jj}=p$ for $j\geq3$ and other entries zero.
Direct Wirtinger differentiation gives
\[
\mathcal DA=\tfrac14p''+q'+r.
\]
For \eqref{eq:stein-target}, $p=a$, $q=-xp/2$, $r=x^2p/4$, hence
\[
\mathcal DA=\tfrac14\{p''-2xp'+(x^2-2)p\}.
\]
Both $p=a$ and $p=ah_0$ satisfy the displayed scalar equation.
This proves both identities. The derivative of the profile is
$\tfrac12(1-\tanh^2s)>0$, and $-1<\tanh s<1$ gives its bounds.
\end{proof}

For a holomorphic map $F:M\to\C^n$ with Jacobian $J_F$,
the pullback of the represented $(n-1,n-1)$-form has matrix
\begin{equation}\label{eq:stein-pullback}
A_F=B_F^*A(F)B_F,\qquad B_F=\adj(J_F^T).
\end{equation}
Holomorphic pullback commutes with $\partial$ and $\bar\partial$
\cite[\S1.7]{Forstneric}; this is also the part concerning smooth forms of
\cite[Theorem~1.1]{SK}. Thus $\mathcal DA_F=0$, and $A_F$ is
positive semidefinite.
\begin{lemma}\label{lem:cofactor-cover}
Suppose $D_1,\ldots,D_N$ span $\C^n$, and $v\ne0$.
For every row $w_0$, there are indices $i_1,\ldots,i_{n-1}$ such that
\[
\adj([w_0;D_{i_1};\ldots;D_{i_{n-1}}]^T)v\ne0.
\]
The semicolons in $[w_0;D_{i_1};\ldots;D_{i_{n-1}}]$ indicate the matrix
formed by stacking the listed rows.
\end{lemma}
\begin{proof}
Extend $v$ to a basis using rows of the spanning family.
The determinant $\det[v;D_{i_1};\ldots;D_{i_{n-1}}]$ is nonzero.
By the adjugate form of Cramer's rule, it is the first component of
the vector in the statement, independently of $w_0$.
\end{proof}
\begin{proof}[Proof of Proposition~\ref{prop:stein}]
Choose global holomorphic functions $V_1,\ldots,V_N$ whose
differentials span every cotangent space, using \cite{EG}.
Separation of two points gives a global holomorphic $W$ with
nonconstant real part: replace a separating function by $-\ii$ times
it if necessary. For every tuple $I=(i_1,\ldots,i_{n-1})$ put
\[
F_I=(W,V_{i_1},\ldots,V_{i_{n-1}}),\qquad Q=\sum_I A_{F_I}.
\]
This is a finite sum; repetitions in a tuple are allowed.
Every summand is smooth, positive semidefinite, and is annihilated by
$\partial\bar\partial$. Lemma~\ref{lem:cofactor-cover}
shows that for each nonzero vector one summand is strictly positive,
so $Q$ is positive definite. All maps have the same first component,
and therefore the same pullback ratio
\[
h=1+\tfrac12\tanh(\Re W).
\]
The pullback of $h_0A$ is $hA_{F_I}$. Hence
$\mathcal DQ=\mathcal D(hQ)=0$. The function $h$ is positive,
bounded by~2, and nonconstant. Lemma~\ref{lem:root} gives a
Gauduchon metric with cofactor $Q$, and multiplying it by
$h^{1/(n-1)}$ gives the required second metric.
\end{proof}

\section{Complex Lie groups}\label{sec:lie}
The abelian classification follows from compact fibers and two transverse
real coordinates. For a nonabelian group, the adjoint representation
provides a solvable quotient, or the radical quotient is semisimple.
In each case a positive nonconstant function gives an explicit failure
of conformal uniqueness.

\subsection{Holomorphic volume forms and group projections}
Let $\Omega$ be a nowhere vanishing holomorphic $(n,0)$-form.
In a chart write $\Omega=\mu\,dz_1\wedge\cdots\wedge dz_n$.
The notation $\iota_X\Omega$ denotes contraction of the differential form
$\Omega$ with the vector field $X$.
For holomorphic vector fields $X_a=\sum_iF_{ia}\partial_{z_i}$ put
$H_{ia}=\mu F_{ia}$. Then
\begin{equation}\label{eq:closed-columns}
d(\iota_{X_a}\Omega)=0\quad\Longleftrightarrow\quad
\sum_i\partial_{z_i}H_{ia}=0.
\end{equation}
We use finite families of these fields which span every tangent space.

\begin{lemma}\label{lem:volume-frame}
Suppose that $H$ is a holomorphic $n\times r$ matrix of rank $n$
satisfying \eqref{eq:closed-columns}. The cofactor matrix
\begin{equation}\label{eq:gram-cofactor}
A_{ji}=\sum_a\overline{H_{ja}}H_{ia}
\end{equation}
is positive Hermitian, satisfies the divergence identities
\eqref{eq:cofactor-divergences},
and satisfies $\mathcal DA=0$. For every smooth real $u$,
\begin{equation}\label{eq:cofactor-harmonic}
\mathcal D(uA)=\sum_{i,j}A_{ji}u_{z_i\bar z_j}.
\end{equation}
Compatible families give a global Gauduchon metric through
Lemma~\ref{lem:root}.
\end{lemma}
\begin{proof}
The quadratic form is $|H^Tv|^2$, positive for $v\ne0$ because
$H^T$ is injective. Since $H$ is holomorphic,
\eqref{eq:closed-columns} gives
\begin{equation}\label{eq:cofactor-divergences}
\sum_i\partial_{z_i}A_{ji}
=\sum_a\overline{H_{ja}}\sum_i\partial_{z_i}H_{ia}=0,
\qquad \sum_j\partial_{\bar z_j}A_{ji}=0.
\end{equation}
Expand $\mathcal D(uA)$ by the product rule. These identities remove
the first order and zeroth order terms, leaving
\eqref{eq:cofactor-harmonic}.
\end{proof}

Every complex Lie group has a right invariant holomorphic volume form
$\Omega_R$. The flow of a left invariant holomorphic vector field is
a right translation, so it preserves $\Omega_R$. Cartan's formula
gives \eqref{eq:closed-columns}. If $L$ and $R$ are the local matrices
of left invariant and right invariant frames, normalized at the identity, then
\begin{equation}\label{eq:native-frame}
H=(\det R)^{-1}L.
\end{equation}
The columns of these matrices are the coordinate coefficients of the
frame vectors, using the same basis of $\mathfrak g$ at the identity.
The coefficient of $\Omega_R$ is $(\det R)^{-1}$.
For $g\in G$, let $c_g(h)=ghg^{-1}$. The adjoint action is
$\Ad(g)=(dc_g)_1:\mathfrak g\to\mathfrak g$, represented in the chosen
basis. The symbol $\det_\C$ is the determinant of a complex linear map,
and $R\Ad(g)$ is ordinary matrix multiplication.
Writing $\Delta(g)=\det_\C\Ad(g)$ gives $L=R\Ad(g)$.
The reconstructed metric is
\[
\omega_{\rm bal}=|\Delta|^{2/(n-1)}\omega_L,
\]
where $\omega_L$ is the left invariant metric determined by the chosen
basis. It is balanced, meaning $d\omega_{\rm bal}^{n-1}=0$.
For unimodular groups, where the modular character $\Delta$ is
constant, this agrees with the invariant balanced construction
\cite[\S3]{FeiYau}.

Let $q:G\to S$ be a surjective holomorphic homomorphism of connected
complex Lie groups, of dimensions $n$ and $m$. Let
$P=dq_1:\mathfrak g\to\mathfrak s$, choose a complex linear
right inverse $T$, and let $B$ be a basis matrix of $\mathfrak s$.
All linear maps are represented in fixed bases. Thus $P$ and $T$ are
$m\times n$ and $n\times m$ matrices, respectively, with $PT=I_m$;
the columns of the invertible matrix $B$ are the chosen basis vectors
of $\mathfrak s$. We write $L_G,R_G$ for the left invariant and right
invariant frame matrices on $G$, and $L_S$ for the left invariant frame
matrix on $S$.
The matrix
\begin{equation}\label{eq:split-columns}
C=[TB,I-TP],\qquad PC=[B,0],\qquad
C\begin{bmatrix}B^{-1}P\\ I\end{bmatrix}=I
\end{equation}
has spanning columns. Thus $H_G=(\det R_G)^{-1}L_GC$ gives the
family in Lemma~\ref{lem:volume-frame}.
In $[TB,I-TP]$, the two blocks are placed side by side; the displayed
column matrix stacks its two blocks vertically. Expressions such as
$L_GC$ and $L_SB$ are ordinary matrix products.
\begin{proposition}\label{prop:group-lift}
Suppose $S$ has a smooth positive nonconstant function $v$ satisfying
\begin{equation}\label{eq:frame-harmonic}
\sum_{i,j,a}(L_SB)_{ia}\overline{(L_SB)_{ja}}v_{z_i\bar z_j}=0.
\end{equation}
Then $G$ has two conformal Gauduchon metrics which are not related
by scaling, with density ratio $u=v\circ q$.
\end{proposition}
\begin{proof}
Here $J_q$ is the Jacobian matrix of $q$ in local holomorphic coordinates.
Let $A_G$ be the cofactor from $H_G$. Differentiating the homomorphism
identity gives $J_qL_G=L_SP$, hence
$J_qH_G=(\det R_G)^{-1}[L_SB,0]$.
The mixed Hessian chain rule for a holomorphic map gives
\[
(v\circ q)_{z_i\bar z_j}
=\sum_{k,\ell}(J_q)_{ki}\overline{(J_q)_{\ell j}}
v_{z_k\bar z_\ell}.
\]
Therefore the contraction in \eqref{eq:cofactor-harmonic} is
$|\det R_G|^{-2}$ times \eqref{eq:frame-harmonic}, and is zero.
The extra columns project to zero. Lemmas~\ref{lem:volume-frame}
and~\ref{lem:root} give the two metrics. Surjectivity makes
$u$ nonconstant.
\end{proof}

\subsection{Nonabelian groups}
\begin{proposition}\label{prop:nonabelian}
Every connected nonabelian complex Lie group of complex dimension
$n\geq2$ has a Hermitian conformal class containing two Gauduchon
metrics which are not related by scaling.
\end{proposition}
The proof of Proposition~\ref{prop:nonabelian} distinguishes the
semisimple quotient, the nilpotent case, and the solvable case which is
not nilpotent.
\begin{lemma}\label{lem:poisson}
For a connected complex semisimple Lie group $S$ of positive dimension,
there are a complex basis $B$ and a smooth nonconstant real function
$v$, $1\leq v\leq2$, satisfying \eqref{eq:frame-harmonic}.
\end{lemma}
\begin{proof}
Choose a compact real form $K_S$. If $\mathfrak k_S$ is its real
Lie algebra, then $\mathfrak s_\R=\mathfrak k_S\oplus\ii\mathfrak k_S$.
Here $\mathfrak s_\R$ is $\mathfrak s$ regarded as a real Lie algebra,
and $\oplus$ denotes a direct sum of real vector spaces.
Take a real basis $Y_a$ of $\mathfrak k_S$, orthonormal for minus
the restricted real Killing form; it is a complex basis of
$\mathfrak s$. Let $X_a$ be the associated real left invariant
fields, and $JX_a$ those associated to $\ii Y_a$.
On functions invariant under right translations by $K_S$, the real Casimir operator
\[
\mathcal C=\sum_a\big((JX_a)^2-X_a^2\big)
\]
induces the invariant Laplacian on $S/K_S$, up to a positive
constant. This symmetric space has nontrivial Furstenberg boundary.
The Poisson transform of a smooth nonconstant boundary function
between one and two is smooth harmonic with the same bounds. Its
boundary limit along a regular Weyl chamber direction recovers
that function, so the transform is nonconstant
\cite{Furstenberg,KnappWilliamson,Johnson}.
Lift it to $S$ to obtain $v$. Since $v$ is invariant under right
translations by $K_S$, we have $X_av=0$. Since $v$ is harmonic,
$\mathcal Cv=0$. Hence
$\sum_a(X_a^2+(JX_a)^2)v=0$.
For a holomorphic vector field $Y=\sum_i f_i\partial_{z_i}$,
with real normalization $Y^\R=Y+\bar Y$, the coefficients $f_i$ are
holomorphic. The differentiated coefficient terms cancel, giving
\[
(Y^\R)^2v+(\ii Y)^\R{}^2v
=4\sum_{i,j}f_i\bar f_jv_{z_i\bar z_j}.
\]
Summing over the basis proves \eqref{eq:frame-harmonic}.
\end{proof}
The solvable radical $\rad\mathfrak g$ is the largest solvable ideal
of $\mathfrak g$.
If $\mathfrak g$ is not solvable, its connected solvable radical is
a closed complex normal subgroup $R$, and $S=G/R$ is a connected
complex semisimple group of positive dimension. Its Lie algebra is
$\mathfrak g/\rad\mathfrak g$, and the quotient map is a
holomorphic submersion \cite[Chapter~10]{SagleWalde}; see also
\cite[\S2]{Gilligan}. Lemma~\ref{lem:poisson} and
Proposition~\ref{prop:group-lift} prove
Proposition~\ref{prop:nonabelian} in this case.

For subspaces $\mathfrak a,\mathfrak b\subset\mathfrak g$,
$[\mathfrak a,\mathfrak b]$ denotes the complex span of the Lie brackets
$[X,Y]$ with $X\in\mathfrak a$ and $Y\in\mathfrak b$.
Suppose next that $\mathfrak g$ is nonabelian and nilpotent. Put
$D=[\mathfrak g,\mathfrak g]$ and $D_3=[\mathfrak g,D]$.
Nilpotence gives $D_3\ne D$. Choose a complex linear functional
$\varphi$ vanishing on $D_3$ and nonzero on $D$, and choose
$H,E$ with $\varphi([H,E])=1$. The connected adjoint group acts
trivially on $D/D_3$. Thus
\[
f_H(g)=\varphi(\Ad(g)H-H),\qquad
f_E(g)=\varphi(\Ad(g)E-E)
\]
are additive holomorphic characters. Indeed,
\[
\Ad(gh)H-H=\Ad(g)(\Ad(h)H-H)+(\Ad(g)H-H),
\]
the first difference belongs to $D$, and $\varphi$ is invariant
there. Their differentials are
$df_H(U)=\varphi([U,H])$ and $df_E(U)=\varphi([U,E])$.
The map $f=(f_E,-f_H):G\to\C^2$ has identity differential on
the span of $H,E$. Its image is an open subgroup of $\C^2$,
so it is surjective.
Choose horizontal fields $X_a$ with $X_af_b=\delta_{ab}$,
and vertical fields for the remaining directions, using
\eqref{eq:split-columns} if desired. Put $\Phi=(e^{f_1},f_2)$.
The Kronecker symbol $\delta_{ab}$ is $1$ when $a=b$ and $0$ otherwise.
In the matrix used in Lemma~\ref{lem:volume-frame}, multiply the
second horizontal column by
$e^{f_1}$. The column remains closed because $X_2f_1=0$.
The two horizontal columns then have equal projected speed for
$\Phi$, while the vertical columns project to zero.
On $\C^2\setminus\{0\}$ the positive kernel
\begin{equation}\label{eq:radial-kernel}
k(z,w)=(|z|^2+|w|^2)^{-1}
\end{equation}
has zero complex Hessian trace: if $s=|z_1|^2+|z_2|^2$, then
$k_{z_i\bar z_j}=2s^{-3}\bar z_i z_j-s^{-2}\delta_{ij}$.
The map $\Phi$ avoids zero, so $u=k\circ\Phi$ is smooth
positive and satisfies \eqref{eq:cofactor-harmonic} with value
zero. The differential of $\Phi$ has rank two, so $u$ is nonconstant.
Lemma~\ref{lem:volume-frame} gives the required pair.

\begin{lemma}\label{lem:affine-quotient}
A finite dimensional complex solvable Lie algebra which is not
nilpotent has a surjective homomorphism onto $\aff(\C)$,
the two dimensional algebra with basis $H,E$ and $[H,E]=E$.
This is the Lie algebra of affine transformations $z\mapsto az+b$ of
the complex line, with $a\ne0$.
\end{lemma}
\begin{proof}
Put $D=[\mathfrak g,\mathfrak g]$ and $E=[D,D]$.
The notation $\operatorname{ad}(x)$ denotes the linear map $y\mapsto[x,y]$.
Lie's theorem makes $D$ nilpotent. If $\mathfrak g/E$ were
nilpotent, each $\operatorname{ad}(x)|_D$ would be nilpotent on
$D/[D,D]$. A derivation of a nilpotent algebra which is
nilpotent on its abelianization is nilpotent: the lower central
graded algebra is generated by the abelianization, so the
induced maps on all bracket layers have only zero eigenvalues.
The finite filtration then gives nilpotence on $D$.
Since $\mathfrak g/D$ is abelian, every $\operatorname{ad}(x)$
would be nilpotent, contrary to Engel's theorem.
Consequently $Q=\mathfrak g/E$ is a nonnilpotent metabelian
algebra. Its actions on $I=[Q,Q]$ commute. Simultaneous
triangularization gives a nonzero covector $b\in I^*$ and
a nonzero character $\chi\in Q^*$ such that
\[
b([x,d])=\chi(x)b(d)\qquad(x\in Q,\ d\in I).
\]
The spaces $I^*$ and $Q^*$ are the complex linear duals of $I$ and $Q$.
If all joint weights were zero, Engel's theorem would make $Q$
nilpotent. Choose $a$ with $\chi(a)=1$ and set
$\beta(x)=b([a,x])$. Jacobi's identity gives
\[
\chi([x,y])=0,\qquad
\beta([x,y])=\chi(x)\beta(y)-\chi(y)\beta(x).
\]
Moreover $\beta(a)=0$ and $\beta|_I=b$.
Hence $(\chi,\beta):Q\to\aff(\C)$ is surjective.
The Lie and Engel theorems used here are recalled in
\cite{Humphreys}.
\end{proof}
For a nonnilpotent solvable $\mathfrak g$, let
$p:\mathfrak g\to\aff(\C)$ be this quotient. The connected
adjoint group preserves its kernel. On the quotient, every
automorphism has the form
$H\mapsto H-bE$, $E\mapsto aE$, $a\ne0$.
These coefficients give global holomorphic functions satisfying
\[
a(gh)=a(g)a(h),\qquad b(gh)=b(g)+a(g)b(h).
\]
Their differential is $(\chi,\beta)$, so
$(a,b):G\to\operatorname{Aff}(\C)=\C^*\ltimes\C$ is a
surjective holomorphic submersion. Here
$(a,b)(a',b')=(aa',b+ab')$.
Here $\C^*=\C\setminus\{0\}$, and $\operatorname{Aff}(\C)$ is the group
of affine transformations $z\mapsto az+b$, $a\in\C^*$, $b\in\C$.
The symbol $\ltimes$ denotes a semidirect product; here $\C^*$ acts on
the additive group $\C$ by multiplication.
Choose horizontal left invariant fields projecting to
$a\partial_a$ and $a\partial_b$ and vertical fields projecting
to zero. Their projected Hermitian matrix is $|a|^2I_2$.
The function $u=(|a|^2+|b|^2)^{-1}$ has zero Hessian
contraction by \eqref{eq:radial-kernel}. It is positive and
nonconstant. Lemma~\ref{lem:volume-frame} therefore gives the two
conformal Gauduchon metrics.
This completes the proof of Proposition~\ref{prop:nonabelian}.

\subsection{Compact subgroups and the abelian case}
We denote the Lie group exponential map by $\exp:\mathfrak g\to G$.
\begin{lemma}\label{lem:central}
For every compact subgroup $K$ of a connected complex Lie group,
$\mathfrak k\cap\ii\mathfrak k$ lies in the center of
$\mathfrak g$. If the group is nonabelian, then
$\dim_\C(\mathfrak k\cap\ii\mathfrak k)\leq n-2$.
\end{lemma}
\begin{proof}
For $X,\ii X\in\mathfrak k$, the elements
$\exp(zX)=\exp((\Re z)X)\exp((\Im z)\ii X)$ belong to
$K$ for all $z\in\C$. The entire matrix valued function
$\Ad(\exp(zX))$ is bounded, hence constant. Its derivative
at zero is $\operatorname{ad}(X)$, so $X$ is central.
An algebra whose center has codimension at most one is
abelian: write every vector as a central vector plus a
multiple of one fixed vector. All brackets vanish.
An abelian Lie algebra gives an abelian connected Lie group
by local commutativity and connectedness.
\end{proof}
For completeness, compact complex tangent directions also lie in
the solvable radical. Thus
\[
\dim_\C(\mathfrak g/\rad\mathfrak g)
\leq n-\dim_\C(\mathfrak k\cap\ii\mathfrak k).
\]
A nonzero semisimple complex Lie algebra has dimension at least
three, because every algebra of dimension at most two is solvable.

Write a connected abelian complex Lie group as $G=\C^n/\Lambda$,
where $\Lambda$ is a discrete subgroup with real independent
generators. Set $W=\Span_\R\Lambda$, $H=W\cap\ii W$,
and $q=\dim_\C H$. Its maximal compact subgroup is $W/\Lambda$,
with complex tangent $H$ \cite[\S2]{Kazama,KKO}.
Noncompactness implies $q\leq n-1$.
Here $\Span_\R\Lambda$ is the real linear span of $\Lambda$, and
$[z]$ denotes the coset $z+\Lambda$ in $\C^n/\Lambda$.
\begin{samepage}
\begin{lemma}\label{lem:period-normal}
After a complex linear coordinate change, one of the following
holds:
\begin{enumerate}
\item $q=n-1$ and $W=\{z_1=0\}$;
\item $q=n-1$ and $W=\{\Re z_1=0\}$;
\item $q\leq n-2$ and $\Re z_1=\Re z_2=0$ on $W$.
\end{enumerate}
In the first two cases, respectively, $[z]\mapsto z_1$ and
$[z]\mapsto\Re z_1$ are proper surjections onto $\C$ and $\R$.
\end{lemma}
\end{samepage}
\begin{proof}
The image of $W$ in $\C^n/H$ is totally real, meaning that
its intersection with its multiple by $\ii$ is zero. Its real
basis is complex linearly independent. Adjoining a complex
basis of $H$ and extending to a basis of $\C^n$ gives
$W=\C^q\oplus\ii\R^k\oplus\{0\}$, $0\leq k\leq n-q$.
This proves the alternatives. In the first two cases the
projection has kernel $W$. Choose a linear section $s$ and
generators $p_1,\ldots,p_r$ of $\Lambda$.
The inverse image of a compact base set $D$ is the image of
\[
\{s(b)+\sum_jt_jp_j:b\in D,\ 0\leq t_j\leq1\},
\]
which is compact. The section gives surjectivity.
\end{proof}
When $q=n-1$, the proper projections in Lemma~\ref{lem:period-normal}
satisfy Corollary~\ref{cor:proper-projection}, which gives the Levi
condition and uniqueness in Theorem~\ref{thm:levi}. When $q\leq n-2$, put
$x=\Re z_1$, $y=\Re z_2$. These functions and the forms
$dz_j$ descend to $G$. Define a metric and density by
\begin{equation}\label{eq:abelian-pair}
G_0=\begin{pmatrix}1&-y\\-y&1+y^2\end{pmatrix}
\oplus I_{n-2},\qquad \rho=e^{-x+y^2/2}.
\end{equation}
For matrices, $\oplus$ denotes the block diagonal sum; the identity
block is absent when $n=2$.
The matrix is positive with determinant one. Its adjugate has
first block $\left(\begin{smallmatrix}1+y^2&y\\y&1\end{smallmatrix}\right)$,
and $\mathcal D(\adj G_0)=0$.
The four nonzero terms of $\mathcal D(\rho\adj G_0)$ are
$\tfrac14(1+y^2)\rho$,
$-\tfrac14(1+y^2)\rho$,
$-\tfrac14(1+y^2)\rho$,
$\tfrac14(1+y^2)\rho$. The other terms vanish, so both
$G_0$ and $\rho^{1/(n-1)}G_0$ are Gauduchon.
The values of $\rho$ at $[0]$ and $[(1,0,\ldots,0)]$ are
$1$ and $e^{-1}$, so scaling does not identify them.

\phantomsection\label{proof:lie-gauduchon}
\begin{proof}[Proof of Theorem~\ref{thm:lie}: conditions (2) and (4)]
If \eqref{eq:lie-criterion} holds, Lemma~\ref{lem:central}
makes $G$ abelian. Lemma~\ref{lem:period-normal} and
Corollary~\ref{cor:proper-projection} give uniqueness.
Conversely, uniqueness in every Hermitian conformal class excludes
nonabelian groups by Proposition~\ref{prop:nonabelian}. In the abelian
case, \eqref{eq:abelian-pair} excludes $q\leq n-2$, so
$q=n-1$.
\end{proof}
For these groups the condition is also equivalent to the assertion that
$P_\alpha$ is critical for every metric: the proper projections prove one
direction, and Proposition~\ref{prop:pinchover} proves the other.
This assertion needs no assumption on the number of ends.

\subsection{Examples}\label{sec:lie-examples}
If $T$ is a compact complex torus of dimension $n-1$, both
$T\times\C$ and $T\times\C^*$ satisfy the criterion.
For the elliptic curve $E=\C/(\Z+\ii\Z)$,
$\C\times E$ has uniqueness in every conformal class,
whereas $(\C^*)^2$ admits \eqref{eq:abelian-pair}.
Both underlying real manifolds are homeomorphic to
$\mathbb T^2\times\R^2$, where $\mathbb T^r=(\R/\Z)^r$.
Likewise $\C\times E^2$ and $(\C^*)^2\times E$ are
homeomorphic to $\mathbb T^4\times\R^2$, but their maximal
compact subgroups have CR dimensions two and one. Only the
first has uniqueness in every Hermitian conformal class.

\begin{example}\label{ex:cousin}
Let
\[
\Lambda=\Z(1,0)+\Z(\ii,\sqrt2)+\Z(0,1),
\qquad G=\C^2/\Lambda.
\]
Then $f([z_1,z_2])=\Im z_2$ is proper and pluriharmonic,
and $G$ has uniqueness in every Hermitian conformal class.
Every holomorphic function on $G$ is constant.
\end{example}
\begin{proof}
For $s\in\R$, $s\bmod\Z$ denotes its class in $\R/\Z$.
The real coordinates
\[
(\Re z_1\bmod\Z,\ \Im z_1\bmod\Z,
(\Re z_2-\sqrt2\Im z_1)\bmod\Z,\ \Im z_2)
\]
identify $G$ with $\mathbb T^3\times\R$, proving properness.
Here $W=\{\Im z_2=0\}$ and $H=\{z_2=0\}$, so
Theorem~\ref{thm:lie} applies. The period criterion
\cite[Proposition~2.5(1)]{Battisti} requires
$k\ii+\ell\sqrt2\notin\Z$ for every nonzero
$(k,\ell)\in\Z^2$, which follows from its imaginary part
if $k\ne0$ and from irrationality if $k=0$.
A translation quotient with no nonconstant holomorphic
functions is called a Cousin group.
\end{proof}
The complex subgroup tangent to $H$ need not be closed in the
maximal compact subgroup. The proof uses compactness of the
whole real fiber.

For the affine group on $\C^2$ with multiplication
$(z,w)(z',w')=(z+z',w+e^z w')$, an explicit pair is given by
\begin{equation}\label{eq:affine-pair}
G_0=\begin{pmatrix}1&-w\\-\bar w&1+|w|^2\end{pmatrix},
\qquad \rho=e^{|w|^2-2\Re z}.
\end{equation}
Its determinant is one. The four terms of
$\mathcal D(\rho\adj G_0)$ are
$(1+|w|^2)\rho$, $-(1+|w|^2)\rho$,
$-(1+|w|^2)\rho$, $(1+|w|^2)\rho$, and cancel.
Thus $G_0$ and $\rho G_0$ are Gauduchon and are not related
by scaling. The same pair on $\C^2$ will show why an intact
fiber is needed in Theorem~\ref{thm:fibers}.
For $\mathrm{SL}(2,\C)$ the function
$u\left(\begin{smallmatrix}a&b\\c&d\end{smallmatrix}\right)
=(|c|^2+|d|^2)^{-2}$ is positive nonconstant and has zero
Hessian contraction for the standard Cartan basis.
Here $\mathrm{SL}(2,\C)$ is the group of complex $2\times2$ matrices
with determinant one.
Lemma~\ref{lem:volume-frame} therefore gives two conformal Gauduchon
metrics which are not related by scaling.

\section{The ends and geometric applications}\label{sec:ends}
We now allow some compact fibers to lose points. The boundary potential
penalizes escape inside those fibers; a single intact fiber then turns
the constant base minimum into a minimum attained in the manifold.

\begin{definition}\label{def:completion}
Let $B$ be a point, $\R$, or a parabolic Riemann surface.
A \emph{compactification with compact fibers over $B$} consists of an
open embedding $M\subset X$ into a Hausdorff space and a proper
continuous map $\bar\pi:X\to B$. Its restriction $\pi:M\to B$
is holomorphic for a curve base and pluriharmonic for a real base.
The closed set $E=X\setminus M$ is called the \emph{added boundary}.
A \emph{plurisubharmonic potential at
the added boundary} is a nonpositive upper semicontinuous
$\Phi:X\to[-\infty,0]$ with $\Phi=-\infty$ on $E$, whose
finite points in $M$ are dense and meet every fiber.
For a point base, require $\Phi|_M$ to be plurisubharmonic on
$N\cap M$ for some open neighborhood $N$ of $E$.
For either nonconstant base, require also that
$U=\{x\in M:\Phi(x)>-\infty\}$ be open, that
$\psi=\Phi|_U$ be smooth, and that some open $N\subset X$
contain both the pole set $\{\Phi=-\infty\}$ and the critical
points of $\pi$, with $\psi$ plurisubharmonic on $N\cap U$.
A point is critical when the differential of $\pi$ vanishes.
A completed fiber over $s_0\in B$ is \emph{intact} if
$E\cap\bar\pi^{-1}(s_0)=\varnothing$, or equivalently
$\bar\pi^{-1}(s_0)\subset M$. The condition that there is an intact fiber
requires at least one such $s_0$.
\end{definition}

\subsection{Fiber minima}
We first treat the singular potential needed for the point base.
\begin{lemma}\label{lem:point}
Suppose $\psi:M\to[-\infty,0]$ is upper semicontinuous and
finite on a dense set, every upper level $\{\psi\geq-R\}$
is compact for $R>0$, and $\psi$ is plurisubharmonic outside
a compact set. Then every positive smooth supersolution
$P_\alpha u\geq0$ is constant.
\end{lemma}
\begin{proof}
Enlarge the exceptional set to a nonempty compact $K$ and set
$m=\min_Ku>0$. For $t>0$, put $c_t=\max\{\psi,-t\}$ and
$Q_t=\{\psi\geq-t\}$. Then $c_t$ is plurisubharmonic on
$M\setminus K$, finite upper semicontinuous, and equals
$-t$ outside $Q_t$. Fix $\varepsilon>0$ and choose
$t>m/\varepsilon$. The function $\varepsilon c_t-u$
is an upper semicontinuous subsolution for $L_\alpha$ on
$M\setminus K$, in the upper test sense. In a relatively
compact open set containing $K\cup Q_t\cup\{x\}$, its
boundary upper values on the complement of $K$ are at most
$-m$. The weak maximum principle gives
\[
m\leq u(x)-\varepsilon c_t(x).
\]
The weak principle applies to these viscosity subsolutions
by \cite{HL,Goffi}; no smooth boundary is needed.
The notation $\varepsilon\downarrow0$ means that $\varepsilon$ tends
to zero through decreasing positive values.
At a finite point of $\psi$, let $\varepsilon\downarrow0$;
density and continuity give $u\geq m$ everywhere.
The attained global minimum on $K$ and the strong minimum
principle make $u$ constant.
\end{proof}

For a nonconstant base use Definition~\ref{def:completion}, put
$\psi=\Phi|_U$ and $T=X\setminus N$.
\begin{lemma}\label{lem:compact-levels}
For each compact $D\subset B$ and $R>0$, the set
\[
\{x\in U:\pi(x)\in D,\ \psi(x)\geq-R\}
\]
is compact. The set $T$ is closed, lies in $U$, is proper over
$B$, and contains no critical point. If a fiber is intact,
$B_0=B\setminus\bar\pi(E)$ is nonempty and open, and
$\pi^{-1}(D)$ is compact for each compact $D\subset B_0$.
\end{lemma}
\begin{proof}
The upper level set is closed in compact $\bar\pi^{-1}(D)$
and misses the pole set. The same applies to
$T\cap\bar\pi^{-1}(D)$. A proper map to the locally compact
base is closed, so $\bar\pi(E)$ is closed. The intact fiber
makes its complement nonempty. Over that complement the
completed compact fiber is already in $M$.
\end{proof}
For a real base set $Dq=q''$; for a curve base set
$Dq=q_{z\bar z}$ in a holomorphic coordinate.
\begin{lemma}\label{lem:penalized}
For a positive smooth supersolution $P_\alpha u\geq0$ and
$\varepsilon>0$, the minimum
\begin{equation}\label{eq:penalized}
v_\varepsilon(s)=\min_{x\in U\cap\pi^{-1}(s)}
\big(u(x)-\varepsilon\psi(x)\big)
\end{equation}
exists and is continuous and nonnegative. On each relatively
compact coordinate interval or disc $D\Subset B$, there is
$A_D\geq0$, independent of $\varepsilon$, such that
$Dv_\varepsilon\leq\varepsilon A_D$ in the lower test sense.
The notation $D\Subset B$ means that the closure of $D$ is compact
and contained in $B$.
\end{lemma}
\begin{proof}
The nonnegative cost $u-\varepsilon\psi$ has compact sublevels
over compact base sets by Lemma~\ref{lem:compact-levels}.
The map $U\to B$ is open and onto. Thus the minimum theorem
\cite{FKV} gives attainment and continuity.
Let $q$ touch $v_\varepsilon$ from below at $s_0$, and choose
a minimizer $x_0$. Then $r=u-q\circ\pi-\varepsilon\psi$
has a local minimum at $x_0$. The Hessian chain rule gives
\begin{equation}\label{eq:base-chain}
L_\alpha(q\circ\pi)=a(Dq\circ\pi),\qquad
a=\tr_\alpha(\ii\partial\pi\wedge\bar\partial\pi),
\end{equation}
where $z\circ\pi$ replaces $\pi$ for a curve base. The
coefficient $a$ is positive at regular points. On $T$ the
minimum inequality gives $aDq+\varepsilon L_\alpha\psi\leq0$.
Compactness of $T\cap\pi^{-1}(\overline D)$ and absence of
critical points bound $|L_\alpha\psi|/a$ there.
Outside $T$, suppose $Dq(s_0)>0$. Nearby, $Dq>0$,
$\psi$ is plurisubharmonic, and $L_\alpha r\leq0$.
The strong minimum principle makes $r$ constant, while at a
nearby regular point \eqref{eq:base-chain} gives
$L_\alpha r<0$, a contradiction. This proves the estimate
also at critical minimizers.
\end{proof}
\begin{lemma}\label{lem:relative-constant}
Under these hypotheses and the intact fiber condition,
every positive smooth supersolution $P_\alpha u\geq0$ is
constant.
\end{lemma}
\begin{proof}
Take $\varepsilon_j=1/(j+1)$ and the lower relaxed limit $V$
of $v_{\varepsilon_j}$ as in \eqref{eq:relaxed}.
The functions are nonnegative and bounded above locally by
$v_1$. Lemma~\ref{lem:penalized} and viscosity stability give
$DV\leq0$. On a parabolic surface every nonnegative superharmonic
function is constant, and on $\R$ every nonnegative concave function is
constant. Thus $V\equiv c$. For each finite point of $\Phi$,
\[
v_{\varepsilon_j}(\pi(x))\leq u(x)-\varepsilon_j\psi(x)
\]
implies $c\leq u(x)$. Density gives $u\geq c$ on $M$.
On $B_0$ the ordinary fiber minimum
$m(s)=\min_{\pi^{-1}(s)}u$ is attained and continuous.
Since $\psi\leq0$, $v_{\varepsilon_j}\geq m$ there, hence
$c\geq m$. The opposite inequality follows from $u\geq c$.
Thus a fiber over $B_0$ contains a global minimum of $u$,
and the strong minimum principle makes $u$ constant.
\end{proof}

\begin{proof}[Proof of Theorem~\ref{thm:fibers}]
For a point base, $X$ is compact. The upper levels of $\Phi$
are compact and miss $E$, while $X\setminus N$ is a compact
subset of $M$. Lemma~\ref{lem:point} applies.
For either nonconstant base use
Lemmas~\ref{lem:compact-levels}--\ref{lem:relative-constant}.
If a smooth function $f$ bounded from above had a positive Levi direction
at every point, choose $\alpha$ by
Lemma~\ref{lem:trace-selection}. For $f\leq C$,
$C+1-f$ would be a nonconstant positive supersolution,
contradicting Lemma~\ref{lem:point} for a point base or
Lemma~\ref{lem:relative-constant} for either of the other bases. Theorem~\ref{thm:levi}
therefore shows that every $P_\alpha$ is critical and gives uniqueness, while
Theorem~\ref{thm:existence} supplies existence.
\end{proof}

\begin{corollary}\label{cor:cover}
Let $p:\widetilde M\to M$ be a surjective local
biholomorphism between connected noncompact complex manifolds of
dimension $n\geq2$. If every conformal class on
$\widetilde M$ has uniqueness up to scaling, then the same
holds on $M$. If every $P_\beta$ on $\widetilde M$ is
critical, then every $P_\alpha$ on $M$ is critical.
\end{corollary}
\begin{proof}
Two conformal Gauduchon metrics pull back to such metrics;
scaling upstairs descends by surjectivity. Existence follows
from Theorem~\ref{thm:existence}. To prove the assertion about critical operators, use
$P_{p^*\alpha}(u\circ p)=(P_\alpha u)\circ p$.
Here $p^*\alpha$ is the pullback metric; the same notation is used
for pullback of differential forms.
Every positive supersolution on $\widetilde M$ is constant by
Proposition~\ref{prop:pinchover}. Surjectivity then makes every positive
supersolution on $M$ constant, so $P_\alpha$ is critical.
\end{proof}

\subsection{Changes on compact sets}
We say that two manifolds are \emph{biholomorphic at infinity}
if a biholomorphism between complements of compact sets
preserves escape from ambient compact sets in both directions.
Equivalently, sequences going to infinity in either manifold
do so after applying the exterior map or its inverse.
The ambient qualification excludes approach to the artificial
inner boundary.
\begin{proposition}\label{prop:infinity}
If two manifolds are biholomorphic at infinity, Gauduchon metrics are
unique up to scaling in every Hermitian conformal class on one manifold
if and only if the same holds on the other. More precisely, suppose the
conformal classes of two metrics agree outside compact sets under this
biholomorphism. Then Gauduchon metrics are unique up to scaling in one
class if and only if they are unique up to scaling in the other.
\end{proposition}
\begin{proof}
First suppose the metrics agree on the identified exterior.
The escape condition gives common cofinal smooth compact
comparison boundaries. The trace operators and their formal
adjoints agree there. By the minimal growth characterization
in \cite[Definition~2.4, Theorem~2.5, and Corollary~2.8]{Pinchover},
the constant $1$ has minimal growth in one filling exactly
when it does in the other. Thus the trace operators are either both critical or both subcritical. If the operators are critical, Propositions~\ref{prop:pinchover}
and~\ref{prop:rays} give uniqueness on both manifolds.
In the subcritical case Proposition~\ref{prop:taylor}
identifies both positive adjoint cones, by positive homogeneous
bijections, with the same cone of exterior solutions with zero boundary values. One is a
ray exactly when the other is. Proposition~\ref{prop:rays}
then proves the assertion for the two fixed conformal classes.
If the exterior metrics are only conformal, extend their
positive ratio smoothly to a filling after shrinking the
exterior. Rescale one metric so that they agree; its conformal
class is preserved. Finally an arbitrary exterior metric
extends by a cutoff combination with a background metric.
Convexity of positive forms preserves positivity. Applying
the assertion for fixed conformal classes in both directions proves the
assertion about every Hermitian conformal class.
\end{proof}
\begin{corollary}\label{cor:modification}
A proper holomorphic modification which is biholomorphic
outside compact sets preserves uniqueness of Gauduchon metrics up to
scaling in every Hermitian conformal class. In particular, this holds
for blowing up a compact
smooth complex submanifold.
\end{corollary}
\begin{proof}
Properness preserves escape under the forward map; compact
images preserve it under the inverse exterior map. Apply
Proposition~\ref{prop:infinity}. The standard blowup has these
properties \cite[\S3.3.3]{Voisin}.
\end{proof}

\begin{corollary}\label{cor:exterior-projection}
Let $K\subset M$ be nonempty compact, set
$E=M\setminus K$, and let $V$ be an open neighborhood of
$\overline E$. Suppose $\pi:V\to B$ is holomorphic for a
parabolic curve base or pluriharmonic for a real base,
is open on $E$, and has dense regular points in $V$.
Assume that $\overline E\cap\pi^{-1}(D)$ is compact in $M$ for every
compact $D\subset B$. Then every Hermitian conformal class
on $M$ has uniqueness up to scaling.
\end{corollary}
\begin{proof}
Normalize a positive supersolution by $\min_Ku=1$.
For $0<k<1$ use the capped minimum
$v(s)=\min\{k,\inf_{E\cap\pi^{-1}(s)}u\}$, with
empty fiber infimum $+\infty$. Sublevels below one are
compact and cannot approach the inner boundary, where
$u\geq1$. The minimum theorem and the contact argument of
Lemma~\ref{lem:penalized}, with zero potential, give
$Dv\leq0$, so $v\equiv c\leq k$. Outside the compact
projection of the inner boundary, the exterior fibers are
compact. If $c<k$, an attained fiber minimum gives a global
minimum of $u$, contradicting its normalization unless it
is constant. Thus $u\geq k$. The notation $k\uparrow1$ means that $k$ tends to one through increasing
values less than one.
Letting $k\uparrow1$ gives
an attained minimum on $K$, so $u$ is constant. Proposition~\ref{prop:pinchover}
shows that $P_\alpha$ and $P_\alpha^*$ are critical, and
Proposition~\ref{prop:rays} proves the assertion.
\end{proof}

\subsection{Compact deletions and partially deleted fibers}
\begin{corollary}\label{cor:deletion}
Let $X$ be a compact connected complex manifold of dimension
$n\geq2$, and let $E\subset X$ be nonempty closed.
Suppose one open neighborhood of $E$ carries a
plurisubharmonic function $\varphi$, finite on a dense set,
with $E\subset\{\varphi=-\infty\}$. Then $X\setminus E$
is connected and noncompact, and every Hermitian conformal
class on it has a Gauduchon metric unique up to scaling.
\end{corollary}
\begin{proof}
The set is locally pluripolar and has empty interior. A
separation of its complement would define a nonconstant
bounded locally constant plurisubharmonic function. Removable
singularities \cite[Theorem~1.8.6]{Forstneric} extend it to
compact $X$, contradicting the maximum principle. The
complement is dense and noncompact. Choose $x_0\notin E$,
restrict $\varphi$ to a smaller region where $\varphi<0$
which avoids $x_0$, and extend it by zero elsewhere.
This extension is nonpositive upper semicontinuous, has
compact upper levels disjoint from $E$, and is
plurisubharmonic near $E$. Apply Theorem~\ref{thm:fibers}
with a point base.
\end{proof}
For a finite set of points use the local potential
$\log\sum_j|z_j|^2$. Its Levi matrix away from zero is
$\delta_{jk}/q-\bar z_jz_k/q^2$, $q=\sum_j|z_j|^2$,
which is positive semidefinite by Cauchy--Schwarz.
Truncation shows that the function is plurisubharmonic at the pole.
If a compact deleted set lies in a coordinate hyperplane
$z_j=a$, use $\log|z_j-a|^2$. This includes a small circle
and a closed complex disc in a coordinate complex line
when $n\geq2$.
For a nonzero effective divisor $D$, with associated line
bundle $\mathcal O(D)$ and canonical section $s_D$, assume
that a smooth Hermitian metric $h$ has
$\ii\Theta_h\leq0$ on an open neighborhood of its support.
In the convention $|e|_h^2=e^{-\eta}$,
$\Theta_h=\ddbar\eta$. Then
$\varphi=\log|s_D|_h^2$ is plurisubharmonic near the support,
since $\ii\ddbar\varphi=-\ii\Theta_h$ off it.
Truncation extends across its zeros. Thus the complement
also satisfies Corollary~\ref{cor:deletion}.
These divisor conventions are recalled in \cite{GH}.

Identify $\PP^1=\C\cup\{\infty\}$. Choose a smooth
$q:\PP^1\to[0,\infty)$, zero only at infinity, equal to
$|t|^2$ near infinity in $t=1/w$, and equal to one off a
slightly larger neighborhood. For $c\geq\max\{1,\max q\}$ put
\begin{equation}\label{eq:fiber-potential}
\chi(w)=\log q(w)-\log c\ (w\ne\infty),\qquad
\chi(\infty)=-\infty.
\end{equation}
It is nonpositive, smooth on $\C$, constant near zero,
and plurisubharmonic outside a compact annulus avoiding zero.

\begin{example}\label{ex:partial}
For any proper closed subset $A\subset\C$,
$M_A=(\C\times\PP^1)\setminus(A\times\{\infty\})$
has a unique Gauduchon metric up to scaling in every
Hermitian conformal class.
\end{example}
\begin{proof}
Use $X=\C\times\PP^1$, $\bar\pi=\operatorname{pr}_1$,
and $\Phi(z,w)=\chi(w)$. Here $\operatorname{pr}_1(z,w)=z$ is the projection onto the first factor.
A fiber over $z_0\notin A$
is intact. The projection is regular, and the potential
is plurisubharmonic near its pole set. Apply
Theorem~\ref{thm:fibers}. Connectedness follows from the
dense open $\C^2$ subset, and projection onto $\C$ gives
noncompactness. If $A=\C$, the result is $\C^2$, where
\eqref{eq:affine-pair} gives nonuniqueness. Thus the
intact fiber assumption cannot be omitted.
\end{proof}
For a compact connected Riemann surface $F$, the projection
$\C\times F\to\C$ is proper, so
Corollary~\ref{cor:proper-projection} gives uniqueness.
Replacing $F$ by another compact curve preserves this
conclusion without identifying the complex ends.

\begin{proposition}\label{prop:no-point-potential}
For $A=\{\Re z\geq0\}$, the surface $M_A$ has no
nonpositive upper semicontinuous function, finite on a
dense set and plurisubharmonic outside a compact set,
whose upper levels are all compact.
\end{proposition}
\begin{proof}
Let $K$ be a compact exceptional set, with projection in
$\{|z|\leq R_0\}$. For $\Re z>R_0$, the fiber is an
entire affine line outside $K$. The restriction of a
proposed potential to it is either identically $-\infty$
or a subharmonic function bounded from above, hence constant.
A finite constant would put the entire closed noncompact
fiber in an upper level. Therefore the potential is
$-\infty$ on a nonempty open set, contradicting density
of its finite points.
\end{proof}
The surface is nevertheless an open subset of
$\PP^1\times\PP^1$: its complement is
$(\{\infty\}\times\PP^1)\cup(\overline A\times\{\infty\})$.

\begin{example}\label{ex:blowups}
Blow up $\C\times\PP^1$ at the closed discrete set
$\{(k,p_0):k=1,2,\ldots\}$. The resulting surface has
uniqueness in every conformal class and cannot be an open
subset of a compact complex surface. The same conclusions
hold for the blowup of the half plane surface $M_A$ at
$\{(k,0):k=1,2,\ldots\}$. In the latter case there is
also no potential of Proposition~\ref{prop:no-point-potential}.
\end{example}
\begin{proof}
The first blowdown is proper: each compact set meets only
finitely many centers. Its composition with projection
to $\C$ is proper, so Corollary~\ref{cor:proper-projection}
applies. Its disjoint exceptional curves have normal
bundle $\mathcal O_{\PP^1}(-1)$ \cite[Chapter~VI, \S3]{Wells}
and self intersection $-1$ \cite[Theorem~11.3]{MS}.
For a compact surface $X$, $H_2(X,\R)$ denotes its second singular
homology group with real coefficients. The class $[E_j]$ is the
homology class of the exceptional curve $E_j$, and the dot denotes
the intersection pairing.
In a compact surface their homology classes would satisfy
$[E_j]\cdot[E_k]=-\delta_{jk}$, so they would be linearly
independent. This contradicts the finite dimension of
$H_2(X,\R)$ \cite[\S5]{Voisin}.
For the second surface blow up the ambient
$\C\times\PP^1$ at the same centers. The completed
projection is proper, and the pullback of \eqref{eq:fiber-potential}
is constant near every center. It is therefore
plurisubharmonic near the projection's critical points
as well as its poles. A fiber over $\Re z<0$ is intact,
so Theorem~\ref{thm:fibers} applies. In a blowup chart
$z-k=ab$, $w=b$, the projection differential is
$b\,da+a\,db$ and vanishes at $(0,0)$, which explains
the critical point requirement. The exceptional curves
again exclude compactification. Finally an open region
$\{\Re z>0,\Im z>R\}$ beyond the projection of any
compact core avoids all integer centers and retains its
affine line fibers. The proof of
Proposition~\ref{prop:no-point-potential} applies there.
\end{proof}

\subsection{Pluriharmonic bases and coverings}
\begin{example}\label{ex:real-base}
Use the Cousin group and proper function $f$ of
Example~\ref{ex:cousin}, and put
$M=(G\times\PP^1)\setminus(\{f\geq0\}\times\{\infty\})$.
Every Hermitian conformal class on $M$ has uniqueness up
to scaling. This threefold has no point base potential
as in Lemma~\ref{lem:point}, and every holomorphic map
from it to a noncompact Riemann surface is constant.
\end{example}
\begin{proof}
Take $X=G\times\PP^1$ over $\R$ with projection $f$ and
potential $\chi(w)$. A negative fiber is intact, and
Theorem~\ref{thm:fibers} applies. Beyond the projection
of any compact core the positive fibers contain entire
affine lines. The argument of
Proposition~\ref{prop:no-point-potential} excludes a
point base potential. A holomorphic function on $M$ is
constant on every full $\PP^1$ fiber over $f<0$ by the
compact maximum principle. The identity theorem makes
it independent of $w$ throughout $M$, and its value on
$w=0$ is constant because $G$ is Cousin.
Global holomorphic functions separate points of every
noncompact Riemann surface \cite[Corollary~26.8]{Forster};
composition therefore excludes nonconstant maps to one.
A proper map to a compact curve would make the completed
space compact and give the point base case already excluded.
\end{proof}

\begin{example}\label{ex:cover-essential}
Let $E=\C/(\Z+\ii\Z)$ and $\eta=[1/2]\in E$.
On $V=(G\times E\times\PP^1)\setminus
(\{|f|\geq1\}\times E\times\{\infty\})$,
the involution $\tau(g,e,w)=(-g,e+\eta,w)$ is free.
The quotient $N=V/\langle\tau\rangle$ has uniqueness in
every conformal class, but admits none of the direct
compactifications of Theorem~\ref{thm:fibers}.
Here $\langle\tau\rangle=\{\mathrm{id},\tau\}$ is the group generated by
the involution $\tau$; its quotient identifies each point with its image
under $\tau$.
\end{example}
\begin{proof}
The potential $\chi(w)$ on the completed product over
$\R$ has an intact fiber over zero, so Theorem~\ref{thm:fibers}
applies to $V$. Since $\eta$ is nonzero of order two, the quotient
map is a holomorphic two sheeted covering. It is proper:
it is closed and has finite fibers. Corollary~\ref{cor:cover}
gives uniqueness on $N$.
Escaping affine fibers exclude point base potentials on
$V$ and, by properness, on $N$. Holomorphic functions on
$V$ are constant: compactness of $E$ and the full
projective fibers on $|f|<1$, followed by the identity
theorem, reduces them to functions on $G$. Thus no
noncompact curve base is available on $N$.
If $\varphi$ is real pluriharmonic on $V$, the same compact
fibers and identity argument make it a function on $G$.
The coefficients of $\partial\varphi$ in $dz_1,dz_2$
are holomorphic functions on $G$, hence constants.
Write $d\varphi=2\Re(c_1dz_1+c_2dz_2)$. Vanishing
period integrals give
\[
\Re c_1=0,\quad \Re c_2=0,\quad
\Re(\ii c_1+\sqrt2c_2)=0.
\]
Hence $\varphi=af+b$. A function lifted from $N$ is
invariant under $\tau$, and $f(-g)=-f(g)$ forces $a=0$.
Therefore every real pluriharmonic function on $N$ is
constant. A compact curve base would give a point base
compactification as above, so that possibility is excluded
as well. The quotient and product atlases are the standard
ones \cite[Chapter~IV, \S5]{FG}.
\end{proof}

\subsection{Line bundles and a rank two counterexample}
\begin{corollary}\label{cor:line-bundle}
Let $B$ be a compact connected complex manifold of dimension
$m\geq1$, and let $L\to B$ be a holomorphic line bundle
with a smooth Hermitian metric of semipositive curvature.
We write $\Tot(L)$ for the total space of $L$, and
$\pi:\Tot(L)\to B$ for its bundle projection.
Then $\Tot(L)$ has a Gauduchon metric unique up to scaling
in every Hermitian conformal class, and every $P_\alpha$
is critical.
\end{corollary}
\begin{proof}
In a frame $|e|^2=e^{-\phi}$, semipositive curvature
means $\ii\ddbar\phi\geq0$. The squared norm
$t=|w|^2e^{-\phi}$ is proper, since its sublevels are
closed disk bundles over compact $B$. For $t>0$,
\[
\ii\ddbar(-\log t)=\pi^*(\ii\ddbar\phi)\geq0.
\]
On the one point compactification set
$\Phi=-\log\max\{1,t\}$ and $\Phi(\infty)=-\infty$.
It is continuous on the total space, plurisubharmonic
on $t>1$, and its upper levels are compact norm disks.
Theorem~\ref{thm:fibers} applies with a point base.
\end{proof}

\begin{example}\label{ex:rank-two}
The threefold $M=\PP^3\setminus\PP^1$, identified with
$\Tot(\mathcal O_{\PP^1}(1)^{\oplus2})$, has two
conformal Gauduchon metrics which are not related by scaling.
\end{example}
\begin{proof}
Write $M=\{[u:v]:u\in\C^2,\ v\in\C^2\setminus\{0\}\}$
and $\pi([u:v])=[v]$. The brackets $[u:v]$ denote homogeneous projective coordinates:
$(u,v)$ is taken up to multiplication by a common nonzero complex scalar.
The symbol $[v]$ denotes the corresponding point of $\PP^1$.
The notation $\operatorname{Hom}(\C v,\C^2)$ denotes the space of complex linear maps
from the line $\C v$ to $\C^2$.
The fiber is
$\operatorname{Hom}(\C v,\C^2)$, which identifies the
bundle. Its zero section is $\{u=0\}\simeq\PP^1$.
Put
\[
t=\frac{|u|^2}{|v|^2},\quad \omega=\pi^*\omega_B,
\quad \omega_B=\ii\ddbar\log(1+|z|^2).
\]
In the affine chart $z=v_1/v_0$, $w=u/v_0$, we have
$t=|w|^2/(1+|z|^2)$; the other chart has
$z'=1/z$, $w'=w/z$. Define
\begin{align}\label{eq:rank-two-data}
A(t)&=\tfrac12(1+t)(2+7t+4t^2),&
f(t)&=\frac{1+2t}{1+t},\notag\\
\psi([v])&=\frac{|v_0|^2}{|v|^2}-\tfrac12,&
h&=1+\tfrac14f(t)\psi,\\
\rho(t)&=A(t)^{-1}\exp\left(-\int_0^t\frac{ds}{A(s)}\right),&
\alpha&=(A(t)+t)\omega+\ii\ddbar t.\notag
\end{align}
These are smooth at the zero section, $A,\rho>0$, and
$3/4<h<5/4$ because $1\leq f<2$ and $|\psi|\leq1/2$.
Functions on the base, such as $\psi$, are viewed as functions on $M$
by pullback through $\pi$.
For $S=1+|z|^2$, set $\theta_a=dw_a-w_a\partial\log S$.
Direct differentiation gives
\[
V=\frac{\ii}{S}\sum_{a=1}^2\theta_a\wedge\bar\theta_a,
\qquad \ii\ddbar t=V-t\omega,\qquad
\partial t=S^{-1}\sum_a\bar w_a\theta_a.
\]
Thus $\alpha=A\omega+V$ is positive; its matrix in a
normal unitary bundle frame is $\operatorname{diag}(A,1,1)$.
The notation $\operatorname{diag}(A,1,1)$ denotes the diagonal matrix
with these three entries.
On the base, $L_{\omega_B}\psi=-2\psi$.
Taking traces gives
\[
|\partial t|_\alpha^2=t,\qquad
L_\alpha t=2-t/A,\qquad L_\alpha\psi=-2\psi/A.
\]
The mixed trace of $\partial t$ and $\partial\psi$ is
zero because these are vertical and horizontal. Therefore
\begin{equation}\label{eq:rank-two-harmonic}
L_\alpha(f(t)\psi)
=\left(tf''+2f'-\frac{tf'+2f}{A}\right)\psi=0.
\end{equation}
Indeed $f'=(1+t)^{-2}$, $f''=-2(1+t)^{-3}$, and
$A(tf''+2f')=tf'+2f$. Hence $L_\alpha h=0$.
Since $\omega$ and $\ii\ddbar t$ are closed,
$d\alpha=(A'+1)dt\wedge\omega$. Their ranks give
$\omega^2=0$, $dt\wedge V^2=0$, and
$dt\wedge\alpha^2=2A\,dt\wedge\omega\wedge V$.
Using $\rho'/\rho=-(A'+1)/A$ yields
\[
d(\rho\alpha^2)
=\left(\rho'+\rho\frac{A'+1}{A}\right)dt\wedge\alpha^2=0.
\]
Thus $\beta=\rho^{1/2}\alpha$ is balanced. For
$\Omega=\rho\alpha^2$, its type and closedness give
$\partial\Omega=\bar\partial\Omega=0$, so
\[
\ddbar(h\Omega)=\ddbar h\wedge\Omega=0
\]
by \eqref{eq:rank-two-harmonic}. Consequently $\beta$
and $h^{1/2}\beta$ are Gauduchon. On the zero section
$h=1+\psi/4$ is nonconstant, which excludes scaling.
\end{proof}
The logarithmic norm explains why the line bundle argument
does not apply here. In these coordinates,
\[
-\log t=\log(1+|z|^2)-\log(|w_1|^2+|w_2|^2),\qquad
\left.\frac{\partial^2(-\log t)}{\partial w_2\partial\bar w_2}
\right|_{(z,w_1,w_2)=(0,1,0)}=-1.
\]
The compact norm disks persist, but the extra fiber direction
changes the Levi sign.

Finally, for the Cousin group of Example~\ref{ex:cousin},
the domain $D=f^{-1}((-1,1))$ fails uniqueness although
$G$ has it. The Euclidean form descends to a flat K\"ahler
metric $\alpha_0$. On $D$, $u=2+f$ is positive nonconstant
pluriharmonic, and $\ddbar(u\alpha_0)=0$.
Thus $\alpha_0|_D$ and $u\alpha_0|_D$ are distinct up to
scaling. The compact level sets are unchanged, but the
pluriharmonic coordinate has bounded range.

\section{Positive solutions of the complex Laplace equation}\label{sec:harmonic}
By \eqref{eq:adjoint}, the density equation \eqref{eq:density} is the
weighted formal adjoint equation for the complex Laplacian. We compare
uniqueness for the density equation with uniqueness of positive solutions
of $L_g u=0$. This completes the proofs of Theorems~\ref{thm:levi}
and~\ref{thm:lie}. We also prove the equivalence for a fixed conformal
class which contains a balanced metric.

\subsection{Balanced conformal classes}
We use the complex Laplacian $L_g$ defined in \eqref{eq:chern}.
If $G=(g_{i\bar j})$ is the coefficient matrix in holomorphic coordinates,
then
\begin{equation}\label{eq:harmonic-operators}
 L_g u=\sum_{i,j}(G^{-1})_{ji}\partial_i\bar\partial_j u,
 \qquad
 D(A)=\sum_{i,j}\partial_i\bar\partial_j A_{ji}.
\end{equation}
Here $\partial_i=\partial/\partial z_i$ and
$\bar\partial_j=\partial/\partial\bar z_j$ are the Wirtinger derivatives;
$D(A)$ denotes the same operator as $\mathcal D A$ in \eqref{eq:divergence}.
All solutions below are smooth and strictly positive. For the relation
between the complex Laplacian and the Laplace--Beltrami operator, see
\cite[Lemma~1.1]{ACS}.

\begin{proposition}\label{prop:harmonic-conformal}
Let $h=vg$, where $v$ is smooth and positive. Then $L_h=v^{-1}L_g$.
Thus $L_h u=0$ and $L_g u=0$ have the same positive solutions.
\end{proposition}
\begin{proof}
Substitute $(vG)^{-1}=v^{-1}G^{-1}$ in
\eqref{eq:harmonic-operators}. Multiplication by the nowhere zero
function $v^{-1}$ preserves the kernel.
\end{proof}

\begin{theorem}\label{thm:harmonic-balanced}
Let $M$ be a nonempty complex manifold of dimension $n\geq2$, and let
$g$ be a Hermitian metric whose conformal class contains a balanced
metric. The Gauduchon metric in $[g]$ is unique up to scaling if and only
if every positive solution of $L_g u=0$ is constant. No condition on the
ends or completeness is required.
\end{theorem}
\begin{proof}
Suppose first that $g$ is balanced. For $A=\adj G$, the equation
$d(g^{n-1})=0$ gives
\begin{equation}\label{eq:harmonic-balanced}
 \sum_j\bar\partial_j A_{ji}=0\quad\text{for each }i,
 \qquad
 \sum_i\partial_i A_{ji}=0\quad\text{for each }j.
\end{equation}
The second identity is the conjugate of the first, using Hermitian
symmetry. Product differentiation now gives
\begin{equation}\label{eq:harmonic-density}
 D(\rho\adj G)=\sum_{i,j}A_{ji}\partial_i\bar\partial_j\rho
              =\det(G)L_g\rho.
\end{equation}
Indeed, the first order terms vanish by
\eqref{eq:harmonic-balanced}, as does $\rho D(A)$.
Since $\det G>0$, a positive function solves the density equation if and
only if it solves $L_g\rho=0$. Proposition~\ref{prop:rays} identifies
Gauduchon metrics in $[g]$ with positive densities by
$h=\rho^{1/(n-1)}g$. Thus uniqueness of these metrics up to scaling is
equivalent to proportionality of all positive solutions of the complex
Laplace equation. Comparison with the solution $1$ proves the assertion.

If $b\in[g]$ is balanced, apply the result to $b$.
Proposition~\ref{prop:harmonic-conformal} identifies the positive solutions,
and $[b]=[g]$ identifies the Gauduchon metrics being compared.
\end{proof}

\begin{proof}[Proof of Theorem~\ref{thm:levi}: condition (2) implies (4)]
If $P_g=-L_g$ is critical, every positive supersolution is constant by
Proposition~\ref{prop:pinchover}. Every positive solution of $L_g u=0$
is such a supersolution, so it is constant.
\end{proof}

\subsection{Manifolds with at least two ends}
The \hyperref[proof:levi-metric-ends]{Gauduchon part of the proof of
Theorem~\ref{thm:levi}} uses positive adjoint solutions. We now use the
Martin boundary to prove the corresponding assertion for positive
solutions of the complex Laplace equation.

\phantomsection\label{proof:levi-positive-ends}
\begin{proof}[Proof of Theorem~\ref{thm:levi}: positive solutions when $M$ has at least two ends]
Fix a Hermitian metric $g$ and put $P=-L_g$. Since $P1=0$, the operator is either critical or
subcritical. If it is critical, every positive supersolution is a
constant multiple of $1$ by \cite[Corollary~2.8(ii)]{Pinchover}.
Every positive solution is therefore constant.

Conversely, suppose every positive solution is constant. If $P$ were
subcritical, its positive Green function would define a Martin
compactification. Distinct ends lie in distinct components of the
Martin boundary, and each such component contains a minimal point;
see \cite[Section~7, Corollaries~4--5, and Section~8, Theorem~6 and
Corollary~3]{TaylorMartin}. Choose points $\xi,\eta$ in two such components
and normalize their Martin kernels at a fixed $x_0\in M$:
\[
 K_P(x_0,\xi)=K_P(x_0,\eta)=1.
\]
The kernels separate Martin boundary points. They cannot be
proportional, since normalization would make the proportionality
constant $1$. This contradicts the assumption that every positive
solution is constant. Thus $P$ is critical. The
\hyperref[proof:levi-metric-ends]{Gauduchon part of the proof} shows that
this is equivalent to uniqueness of the Gauduchon metric in $[g]$.
Applying these equivalences to every metric proves that (4) is equivalent
to (1), (2), and (3) when $M$ has at least two ends.
\end{proof}

\subsection{Translation quotients}
For abelian groups, we use a nonconstant positive solution of the
complex Laplace equation. Put $x=\operatorname{Re}z_1$, $y=\operatorname{Re}z_2$ and
\begin{equation}\label{eq:harmonic-height}
 G(y)=\begin{pmatrix}1&-y\\-y&1+y^2\end{pmatrix},\quad
 A(y)=G(y)^{-1}=\begin{pmatrix}1+y^2&y\\y&1\end{pmatrix},\quad
 u(x,y)=e^{x-y^2/2}.
\end{equation}
Both matrices are positive definite and have determinant $1$.
In dimension $2+m$, use $G(y)\oplus I_m$.

\begin{lemma}\label{lem:harmonic-height}
The function $u$ in \eqref{eq:harmonic-height} is a positive solution of
$L_g u=0$ for the metric $G(y)\oplus I_m$.
\end{lemma}
\begin{proof}
Its complex Hessian in the first two coordinates is
\begin{equation}\label{eq:harmonic-hessian}
 (\partial_i\bar\partial_j u)_{1\leq i,j\leq2}
  =\frac{u}{4}\begin{pmatrix}1&-y\\-y&y^2-1\end{pmatrix}.
\end{equation}
Contraction with $A(y)$ gives
\[
 \frac{u}{4}\big((1+y^2)-2y^2+(y^2-1)\big)=0.
\]
The derivatives in the remaining coordinates vanish.
\end{proof}

\begin{proposition}\label{prop:harmonic-quotient}
Let $X=\C^{2+m}/\Lambda$ have its translation complex structure, and
suppose each period has zero real part in its first two coordinates.
There is a Hermitian metric on $X$ for which the complex Laplace
equation has a nonconstant positive solution.
\end{proposition}
\begin{proof}
The functions $x,y$ descend to $X$, so the matrix $G(y)\oplus I_m$ and
the function $u=e^{x-y^2/2}$ descend. Translation chart changes have
identity Jacobian, and these coefficients define a Hermitian metric.
Lemma~\ref{lem:harmonic-height} proves the equation in every chart.
At the images of $0$ and $(1,0,\ldots,0)$ the function has values $1$
and $e$, respectively. It is therefore nonconstant.
\end{proof}

This is the metric in \eqref{eq:abelian-pair}. There the density $\rho=e^{-x+y^2/2}$ satisfies
$D(A)=D(\rho A)=0$ and gives a Gauduchon pair which is not related by
scaling. The positive solution above is $u=\rho^{-1}$; its equation
follows from \eqref{eq:harmonic-hessian}.

\begin{lemma}\label{lem:harmonic-descent}
Let $F:N\to M$ be a surjective local biholomorphism between complex
manifolds of the same dimension. If, for every Hermitian metric on $N$, every
positive solution of the complex Laplace equation is constant, then the
same holds on $M$.
\end{lemma}
\begin{proof}
For a Hermitian metric $g$ on $M$, its pullback is a Hermitian metric
on $N$. The holomorphic chain rule gives
\[
 L_{F^*g}(u\circ F)=(L_g u)\circ F.
\]
Thus a positive solution of $L_g u=0$ pulls back to a positive solution
on $N$. It is constant by hypothesis, and surjectivity makes
$u$ constant.
\end{proof}

\begin{samepage}
\begin{theorem}\label{thm:harmonic-abelian}
Let $X=\C^n/\Lambda$ be a connected noncompact complex abelian Lie
group, with $n\geq2$. Set $W=\Span_{\R}\Lambda$.
The following conditions are equivalent:
\begin{enumerate}
\item For every Hermitian metric on $X$, every positive solution of the
complex Laplace equation is constant.
\item $\dim_{\C}(W\cap iW)=n-1$.
\end{enumerate}
\end{theorem}
\end{samepage}
\begin{proof}
Use the period normal form in Lemma~\ref{lem:period-normal}.
If the periods have zero real part in two coordinates, then
$\dim_{\C}(W\cap iW)\leq n-2$, and
Proposition~\ref{prop:harmonic-quotient} supplies a nonconstant positive
solution. In the other two cases $W$ is the kernel of a complex
coordinate, or the kernel of the real part of one. Its complex part has
dimension $n-1$.

In these two cases the coordinate projection has compact fibers, with
base $\C$ or $\R$. The proof of Corollary~\ref{cor:proper-projection}
shows that every positive supersolution is constant for every metric.
It uses the fact that $\C$ is parabolic in the first case and that every
nonnegative concave function on $\R$ is constant in the second.
The attained fiber minimum is a global minimum, so the strong minimum
principle makes the function constant on $X$. This applies in particular
to positive solutions of the complex Laplace equation.

The normalized and original presentations are related by surjective
local biholomorphisms in both directions. Apply
Lemma~\ref{lem:harmonic-descent} to both maps. The coordinate change is
complex linear and preserves $\dim_{\C}(W\cap iW)$.
\end{proof}

\subsection{Complex Lie groups}
The construction in Lemma~\ref{lem:volume-frame} gives a balanced
first metric. This observation lets us apply
Theorem~\ref{thm:harmonic-balanced} to its nonabelian examples.

\begin{lemma}\label{lem:harmonic-closed}
Every connected noncommutative complex Lie group of dimension $n\geq2$
has a balanced metric $g$ and a conformal Gauduchon metric $h$ which is
not a constant multiple of $g$.
\end{lemma}
\begin{proof}
For the closed holomorphic $(n-1,0)$-forms in
\eqref{eq:closed-columns}, let $F_a$ denote their vector coefficients.
Since these forms are closed, $\sum_i\partial_iF_a^i=0$. The cofactor density is
\[
 Q_{ji}=\sum_a\overline{F_a^j}F_a^i.
\]
Since the coefficients $F_a^i$ are holomorphic,
\[
 \sum_i\partial_iQ_{ji}
  =\sum_a\overline{F_a^j}\sum_i\partial_iF_a^i=0.
\]
Conjugation gives the other divergence identity in
\eqref{eq:harmonic-balanced}. The first metric, whose adjugate is $Q$,
is therefore balanced.

The proof of Proposition~\ref{prop:nonabelian} uses
Lemma~\ref{lem:volume-frame} for characters, affine quotients,
nilpotent groups, and semisimple groups. In each case, the calculation
above shows that the first metric is balanced. The positive multipliers
constructed there give a conformal Gauduchon metric which is not a
constant multiple of the first metric.
\end{proof}

\begin{proposition}\label{prop:harmonic-nonabelian}
On every connected noncommutative complex Lie group of dimension
$n\geq2$, there is a Hermitian metric for which the complex Laplace
equation has a nonconstant positive solution.
\end{proposition}
\begin{proof}
For the pair $g,h$ in Lemma~\ref{lem:harmonic-closed}, uniqueness in
$[g]$ fails. Since $g$ is balanced,
Theorem~\ref{thm:harmonic-balanced} implies that $L_g u=0$ has a
nonconstant positive solution.
\end{proof}

\begin{proof}[Proof of Theorem~\ref{thm:lie}: the remaining equivalences]
Condition (1) forces commutativity by
Proposition~\ref{prop:harmonic-nonabelian}. The abelian translation
presentation gives compact complex tangent dimension
$\dim_{\C}(W\cap iW)$. Theorem~\ref{thm:harmonic-abelian} gives (4).

Conversely, (4) implies commutativity by Lemma~\ref{lem:central}. Apply the abelian translation presentation and
Theorem~\ref{thm:harmonic-abelian} to obtain (1).
The equivalence of (2) and (4) was proved in the
\hyperref[proof:lie-gauduchon]{Gauduchon part of the proof}.

The implication (1)$\Rightarrow$(4) in Theorem~\ref{thm:levi}
gives (3)$\Rightarrow$(1) here. Under (4), Lemma~\ref{lem:period-normal}
and Corollary~\ref{cor:proper-projection} show that every positive
supersolution is constant for every metric. Theorem~\ref{thm:levi}
then gives (3). Conjugacy of maximal compact subgroups
\cite[p.~47]{KKO} identifies the dimension for the specified $K$.
\end{proof}

\raggedbottom

\end{document}